\documentclass[submission,copyright,creativecommons]{eptcs}
\providecommand{\event}{NCL 2024} % Name of the event you are submitting to

\usepackage{iftex}

  \usepackage{underscore}         % Only needed if you use pdflatex.
  \usepackage[T1]{fontenc}        % Recommended with pdflatex
  \usepackage{breakurl}           % Not needed if you use pdflatex only.
\usepackage{amsmath}
\usepackage{amsthm}
\usepackage{amsfonts}
\usepackage{amssymb}
\usepackage{enumerate}
\usepackage{proof}
\usepackage{quiver}
\usepackage{thmtools, thm-restate}
\usepackage[only,llbracket,rrbracket]{stmaryrd}
\newtheorem{theorem}{Theorem}[section]
\newtheorem{corollary}[theorem]{Corollary}
\newtheorem{lemma}[theorem]{Lemma}

\theoremstyle{definition}
\newtheorem{defn}[theorem]{Definition}
\newtheorem{example}[theorem]{Example}

\makeatletter
\newsavebox{\@brx}

\newcommand{\llangle}[1][]{\savebox{\@brx}{\(\m@th{#1\langle}\)}%
  \mathopen{\copy\@brx\kern-0.5\wd\@brx\usebox{\@brx}}}
\newcommand{\rrangle}[1][]{\savebox{\@brx}{\(\m@th{#1\rangle}\)}%
  \mathclose{\copy\@brx\kern-0.5\wd\@brx\usebox{\@brx}}}
\makeatother
\newcommand{\ldbc}{\llbracket}
\newcommand{\rdbc}{\rrbracket}

\newcommand{\GG}{\Gamma}
\newcommand{\Gg}{\gamma}
\newcommand{\GD}{\Delta}
\newcommand{\Gd}{\delta}
\newcommand{\GL}{\Lambda}

\newcommand{\vd}{\vdash}
\newcommand{\scut}{\mathsf{scut}}
\newcommand{\tl}{\otimes \mathsf{L}}
\newcommand{\tr}{\otimes\mathsf{R}}
\newcommand{\tll}{\otimes^{\mathsf{L}} \mathsf{L}}
\newcommand{\tlr}{\otimes^{\mathsf{R}} \mathsf{L}}
\newcommand{\trl}{\otimes^{\mathsf{L}} \mathsf{R}}
\newcommand{\trr}{\otimes^{\mathsf{R}} \mathsf{R}}
\newcommand{\pass}{\mathsf{pass}}
\newcommand{\unitl}{\mathsf{IL}}
\newcommand{\unitr}{\mathsf{IR}}
\newcommand{\ax}{\mathsf{ax}}
\newcommand{\id}{\mathsf{id}}
\newcommand{\ot}{\otimes}
\newcommand{\otl}{\otimes^{\mathsf{L}}}
\newcommand{\otr}{\otimes^{\mathsf{R}}}
\newcommand{\ol}{\mathbin{\diagup}}
\newcommand{\lo}{\mathbin{\diagdown}}
\newcommand{\lolli}{\multimap}
\newcommand{\lleft}{{\lolli}\mathsf{L}}
\newcommand{\lright}{{\lolli}\mathsf{R}}
\newcommand{\llolli}{\multimap^{\mathsf{L}}}
\newcommand{\rlolli}{\multimap^{\mathsf{R}}}
\newcommand{\llleft}{{\llolli}\mathsf{L}}
\newcommand{\rlleft}{{\rlolli}\mathsf{L}}
\newcommand{\llright}{{\llolli}\mathsf{R}}
\newcommand{\rlright}{{\rlolli}\mathsf{R}}

\newcommand{\I}{\mathsf{I}}

\newcommand{\comp}{\mathsf{comp}}

\newcommand{\comm}{\ot\mathsf{comm}}
\newcommand{\assl}{\mathsf{assoc}^{\mathsf{L}}}
\newcommand{\assr}{\mathsf{assoc}^{\mathsf{R}}}
\newcommand{\vdG}{\vdash_{\mathsf{G}}}
\newcommand{\vdL}{\vdash_{\mathsf{L}}}
\newcommand{\vdT}{\vdash_{\mathsf{T}}}

\newcommand{\darr}{{\downarrow}}

\newcommand{\NL}{$\mathtt{NL}$}
\newcommand{\NMILL}{$\mathtt{NMILL}$}

\newcommand{\LSkNL}{$\mathtt{LSkNL}$}
\newcommand{\LSkG}{$\mathtt{LSkG}$}

\newcommand{\LSkT}{$\mathtt{LSkT}$}
\newcommand{\RSkT}{$\mathtt{RSkT}$}

\newcommand{\SkBiC}{$\mathsf{SkMBiC}$}
\newcommand{\FSkBiC}{$\mathsf{FSkMBiC(At)}$}
\newcommand{\SkBiCT}{$\mathtt{SkMBiCT}$}
\newcommand{\SkBiCA}{$\mathtt{SkMBiCA}$}
\newcommand{\mc}[1]{\mathcal{#1}}
\newcommand{\mf}[1]{\mathsf{#1}}
\newcommand{\mbb}[1]{\mathbb{#1}}

\title{Semi-Substructural Logics {\`a} la Lambek}
\author{Cheng-Syuan Wan
\institute{Department of Software Science
\\
Tallinn University of Technology
\\
Tallinn, Estonia}
\email{cswan@cs.ioc.ee}
}
\def\titlerunning{Semi-Substructural Logics {\`a} la Lambek}
\def\authorrunning{C.-S. Wan}
\begin{document}
\maketitle
% \section{Introduction}\label{intro}
\begin{abstract}
This work studies the proof theory of left (right) skew monoidal closed categories and skew monoidal bi-closed categories from the perspective of  non-associative Lambek calculus.
Skew monoidal closed categories represent a relaxed version of monoidal closed categories, where the structural laws are not invertible; instead, they are natural transformations with a specific orientation. 
Uustalu et al. used sequents with stoup (the leftmost position of an antecedent that can be either empty or a single formula) to deductively model left skew monoidal closed categories, yielding results regarding proof identities and categorical coherence.
However, their syntax does not work well when modeling right skew monoidal closed and skew monoidal bi-closed categories.

We solve the problem by constructing cut-free sequent calculi for left skew monoidal closed and skew monoidal bi-closed categories, reminiscent of non-associative Lambek calculus, with trees as antecedents.
Each calculus is respectively equivalent to the sequent calculus with stoup (for left skew monoidal categories) and the axiomatic calculus (for skew monoidal bi-closed categories).
Moreover, we prove that the latter calculus is sound and complete with respect to its relational models.
We also prove a correspondence between frame conditions and structural laws, providing an algebraic way to understand the relationship between the left and right skew monoidal (closed) categories.
\end{abstract}

\section{Introduction}\label{sec:intro}
Substructural logics are logic systems that lack at least one of the structural rules, weakening, contraction, and exchange.
Joachim Lambek's syntactic calculus \cite{lambek:mathematics:58} is a well-known example that disallows weakening, contraction, and exchange.
Another example, linear logic, proposed by Jean-Yves Girard \cite{girard:linear:87}, is a substructural logic in which weakening and contraction are in general disallowed but can be recovered for some formulae via modalities.
Substructural logics have been found in numerous applications from computational analysis of natural languages to the development of resource-sensitive programming languages.

\emph{Left skew monoidal categories} \cite{szlachanyi:skew-monoidal:2012} are a weaker variant of MacLane's monoidal categories where the structural morphisms of associativity and unitality are not required to be bidirectional, they are natural transformations with a particular orientation.
Therefore, they can be seen as \emph{semi-associative} and \emph{semi-unital} variants of monoidal categories. 
Left skew monoidal categories arise naturally in the semantics of programming languages \cite{altenkirch:monads:2014}, while the concept of semi-associativity is connected with combinatorial structures like the Tamari lattice and Stasheff associahedra \cite{zeilberger:semiassociative:19,moortgat:tamari:20}.

In recent years, Tarmo Uustalu, Niccolò Veltri, and Noam Zeilberger started a research project on \emph{semi-substructural} logics, which is inspired by a series of developments on left skew monoidal categories and related variants by Szlach{\'a}nyi, Street, Bourke, Lack and others \cite{szlachanyi:skew-monoidal:2012,lack:skew:2012,street:skew-closed:2013,lack:triangulations:2014,buckley:catalan:2015,bourke:skew:2017,bourke:skew:2018,bourke:lack:braided:2020}.

We call the languages of left skew monoidal categories and their variants \emph{semi-substructural} logics, because they are intermediate logics between (certain fragments of) non-associative and associative intuitionistic linear logic (or Lambek calculus).
Semi-associativity and semi-unitality are encoded as follows.
Sequents are in the form $S \mid \Gamma \vdash A$, where the antecedent consists of an optional formula $S$, called stoup, adapted from Girard \cite{girard:constructive:91}, and an ordered list of formulae $\Gamma$.
The succedent is a single formula $A$.
We restrict the application of introduction rules in an appropriate way to allow only one of the directions of associativity and unitality.

This approach has successfully captured languages for a variety of categories, including  $(i)$ left skew semigroup \cite{zeilberger:semiassociative:19}, $(ii)$ left skew monoidal \cite{uustalu:sequent:2021}, $(iii)$ left skew (prounital) closed \cite{uustalu:deductive:nodate}, $(iv)$ left skew monoidal closed categories \cite{UVW:protsn,veltri:multifocus:23}, and $(v)$ left distributive skew monoidal categories with finite products and coproducts \cite{VW:2023} through skew variants of the fragments of non-commutative intuitionistic linear logic consisting of combinations of connectives $(\I,\ot,\lolli,\land,\lor)$.
Additionally, discussions have covered partial normality conditions, in which one or more structural morphisms are allowed to have an inverse \cite{uustalu:proof:nodate}, as well as extensions with skew exchange à la Bourke and Lack \cite{veltri:coherence:2021,VW:2023}.

In all of the aforementioned works, internal languages of left skew monoidal categories and their variants are characterized in a similar way which we call sequent calculus {\`a} la Girard.
These calculi with sequents of the form $S \mid \GG \vd A$ are cut-free and by their rule design, they are decidable.
Moreover, they all admit sound and complete subcalculi inspired by Andreoli's focusing \cite{andreoli:logic:1992} in which
rules are restricted to be applied in a specific order.
A focused calculus provides an algorithm to solve both the proof identity problems for its non-focused calculus and coherence problems for its corresponding variant of left skew monoidal category.

By reversing all structural morphisms and modifying coherence conditions in left skew monoidal closed categories, right skew monoidal closed categories emerge \cite{uustalu:eilenberg-kelly:2020}.
Moreover, skew monoidal bi-closed categories are defined by appropriately integrating left and right skew monoidal closed structures.
It is natural for us to consider sound sequent calculi for these categories.
However, the implication rules are not well-behaved when just modeling right skew monoidal closed categories with sequent calculus {\`a} la Girard.

The problem stems from the skew structure concealed within the flat antecedent of $S \mid \GG \vd A$.
While the antecedent $S \mid \GG$ is defined similarly to an ordered list, it is actually a tree associating to the left.
We start in Section \ref{sec:syntax}, by introducing the sequent calculus {\` a} la Girard (\LSkG) for left skew monoidal closed categories from \cite{UVW:protsn} and its equivalent sequent calculus {\` a} la Lambek (\LSkT), which is inspired by sequent calculus 
for non-associative Lambek calculus \cite{bulinska:2009,moot:categorial:2012} with trees as antecedents.

In Section \ref{sec:skew:categories}, we introduce definitions of left (right) skew monoidal closed categories and skew monoidal bi-closed categories, and normality conditions for skew categories.
In Section \ref{sec:calculi:skbic}, we describe two calculi that characterize skew monoidal bi-closed categories: one is an axiomatic calculus (\SkBiCA), while the other is a sequent calculus (\SkBiCT) similar to the multimodal non-associative Lambek calculus \cite{moortgat:multimodl:1996}.
In Section \ref{sec:algebraic:relational:model}, we introduce the relational semantics for \SkBiCA~via preordered sets of possible worlds with ternary relations. 
Furthermore, we show a correspondence theorem (Theorem \ref{thm:main}) between conditions on ternary relations and structural laws on any frame.
The theorem allows us to prove a thin version of main theorems in \cite{uustalu:eilenberg-kelly:2020}.

\section{Sequent Calculus}\label{sec:syntax}
We recall the sequent calculus {\`a la} Girard for left skew monoidal closed categories from \cite{UVW:protsn}, which is a skew variant of non-commutative multiplicative intuitionistic linear logic.

Formulae ($\mf{Fma}$) in \LSkG~are inductively generated by the grammar $A, B::= X \ | \ \I \ | \ A \ot B \ | \ A \lolli B$, where $X$ comes from a set $\mathsf{At}$ of atoms, $\I$ is a multiplicative unit, $\ot$ is multiplicative conjunction and $\lolli$ is a linear implication.
\\
A sequent is a triple of the form $S \mid \Gamma \vdG A$, where the antecedent splits into: an optional formula $S$, called \emph{stoup} \cite{girard:constructive:91}, and an ordered list of formulae $\Gamma$ and succedent $A$ is a single formula.
The symbol $S$ consistently denotes a stoup, meaning $S$ can either be a single formula or empty, indicated as $S = {-}$; furthermore, $X$, $Y$, and $Z$ always represent atomic formulae.
\begin{defn}\label{eq:seqcalc:skmc:Gir}
 Derivations in \LSkG~are generated recursively by the following rules:
\begin{equation*}
	  \begin{array}{c}
		\infer[\ax]{A \mid \quad \vdG A}{}
		\quad
		\infer[\lleft]{A \lolli B \mid \Gamma , \Delta \vdG C}{
		  {-} \mid \Gamma \vdG A
		  &
		  B \mid \Delta \vdG C
		}
		\quad
		\infer[\unitl]{\I \mid \Gamma \vdG C}{{-} \mid \Gamma \vdG C}
		\quad
		\infer[\tl]{A \ot B \mid \Gamma \vdG C}{A \mid B , \Gamma \vdG C}
		\\[5pt]
    \infer[\pass]{{-} \mid A , \Gamma \vdG C}{A \mid \Gamma \vdG C}
		\quad
		\infer[\lright]{S \mid \Gamma \vdG A \lolli B}{S \mid \Gamma , A \vdG B}
		\quad
		\infer[\unitr]{{-} \mid \quad \vdG \I}{}
		\quad
		\infer[\tr]{S \mid \Gamma , \Delta \vdG A \ot B}{
		  S \mid \Gamma \vdG A
		  &
		  {-} \mid \Delta \vdG B
		}
	  \end{array}
	\end{equation*} 
\end{defn}

The inference rules of \LSkG~are similar to the ones in the sequent calculus for non-commutative multiplicative intuitionistic linear logic (\NMILL) \cite{abrusci:noncommutative:1990}, but with some crucial differences: 
\begin{enumerate}
\item The left logical rules $\unitl$, $\tl$ and $\lleft$, read bottom-up, are only allowed to be applied on the formula in the stoup position.
\item The right tensor rule $\tr$, read bottom-up, splits the antecedent of a sequent $S \mid \Gamma, \Delta \vdG A \ot B$ and in the case where $S$ is a formula, $S$ is always moved to the stoup of the left premise, even if $\Gamma$ is empty.
\item The presence of the stoup distinguishes two types of antecedents, $A \mid \Gamma$ and ${-} \mid A, \Gamma$. The structural rule $\pass$ (for `passivation'), read bottom-up, allows the moving of the leftmost formula in the context to the stoup position whenever the stoup is empty.
\item The logical connectives of \NMILL\ (and associative Lambek calculus) typically include two ordered implications $\lo$ and $\ol$, which are two variants of linear implication arising from the removal of the exchange rule from intuitionistic linear logic. In \LSkG, only the right residuation ($B \ol A = A \lolli B$) of Lambek calculus is present. 
% It is currently not clear to us whether the inclusion of the second implication to our logic is a meaningful addition and whether it corresponds to some particular categorical notion.
\end{enumerate}
% The restrictions in 1--4 are essential for precisely capturing all the features of skew monoidal closed categories and nothing more, as we discuss in Section \ref{sec:catsem}.
% Notice also that, similarly to the case of \NMILL, all structural rules of exchange, contraction, and weakening are absent. We give names to derivations and we write $f: S \mid \Gamma \vdash A$ when $f$ is a particular derivation of the sequent $S \mid \Gamma \vdash A$.
For a more detailed explanation and a linear logical interpretation of \LSkG, see \cite[Section 2]{UVW:protsn}.

\begin{theorem}
  \LSkG~is cut-free, i.e. the rules
  \begin{displaymath}
  \begin{array}{c}
    \infer[\mathsf{scut}]{S \mid \Gamma , \Delta \vdG C}{
		  \deduce{S \mid \Gamma \vdG A}{f}
		  &
		  \deduce{A \mid \Delta \vdG C}{g}
	}
		\qquad
	\infer[\mathsf{ccut}]{S \mid \Delta_0 , \Gamma , \Delta_1 \vdG C}{
		  \deduce{{-} \mid \Gamma \vdG A}{f}
		  &
		  \deduce{S \mid \Delta_0 , A , \Delta_1 \vdG C}{g}
    }
  \end{array}
\end{displaymath}
are admissible.
\end{theorem}
\begin{proof}
The proof proceeds by induction on the height of derivations and the complexity of cut formulae.
Specifically, for $\mf{scut}$, we first perform induction on the left premise $f$, and if necessary, we perform subinduction on $g$ or the complexity of the cut formula $A$.
For $\mf{ccut}$, we start by performing induction on the right premise $g$ instead.
The cases other than $\lleft$ and $\lright$ have been discussed in \cite[Lemma 5]{uustalu:sequent:2021}, so we will only elaborate on the cases of $\lolli$.
\\
We first deal with $\scut$.
If $f = \lleft (f' , f'')$, then we permute $\mf{scut}$ up, i.e.
      \begin{displaymath}
        \small\begin{array}{c}
          \begin{array}{c}
            \infer[\mf{scut}]{A' \lolli B' \mid \GG , \GD , \GL \vdG C}{
              \infer[\lleft]{A' \lolli B' \mid \GG , \GD \vdG A}{
                \deduce{{-} \mid \GG \vdG A'}{f'}
                &
                \deduce{B' \mid \GD \vdG A}{f''}
              }
              &
              \deduce{A \mid \GL \vdG C}{g}
            }
          \end{array}
          \quad
          \mapsto
          \quad
          \begin{array}{c}
            \infer[\lleft]{A' \lolli B' \mid \GG , \GD , \GL \vdG C}{
              \deduce{{-} \mid \GG \vdG A'}{f'}
              &
              \infer[\mf{scut}]{B' \mid \GD , \GL \vdG C}{
                \deduce{B' \mid \GD \vdG A}{f''}
                &
                \deduce{A \mid \GL \vdG C}{g}
              }
            }
          \end{array}
        \end{array}
      \end{displaymath}
      If $f = \lright \ f'$, then we perform a subinduction on $g$:
      \begin{itemize}
        \item[--] If $g = \lleft (g' , g'')$, then 
        \begin{displaymath}
          \scriptsize\begin{array}{c}
            \begin{array}{c}
              \infer[\scut]{S \mid \GG , \GD , \GL \vdG C}{
                \infer[\lright]{S \mid \GG \vdG A \lolli B}{
                  \deduce{S \mid \GG , A \vdG B}{f'}
                }
                &
                \infer[\lleft]{A \lolli B \mid \GD , \GL \vdG C}{
                  \deduce{{-} \mid \GD  \vdG A}{g'}
                  &
                  \deduce{B \mid \GL \vdG C}{g''}
                }
              }
            \end{array}
            \quad
            \mapsto
            \quad
            \begin{array}{c}
             \infer[\mf{ccut}]{S \mid \GG , \GD , \GL \vdG C}{
              \deduce{{-} \mid \GD \vdG A}{g'}
              &
              \infer[\scut]{S \mid \GG , A , \GL \vdG C}{
                \deduce{S \mid \GG , A \vdG B}{f'}
                &
                \deduce{B \mid \GL \vdG C}{g''}
              }
             }
            \end{array}
          \end{array}
        \end{displaymath}
        where the complexity of the cut formulae is reduced.
        \item[--] For other rules, we permute $\mf{scut}$ up. For example, if $g = \lright \ g'$, then
        \begin{displaymath}
          \scriptsize\begin{array}{c}
            \begin{array}{c}
              \infer[\scut]{S \mid \GG , \GD \vdG A' \lolli B'}{
                \infer[\lright]{S \mid \GG \vdG A \lolli B}{
                  \deduce{S \mid \GG , A \vdG B}{f'}
                }
                &
                \infer[\lright]{A \lolli B \mid \GD \vdG A' \lolli B'}{
                  \deduce{A \lolli B \mid \GD , A' \vdG B'}{g'}
                }
              }
            \end{array}
            \quad 
            \mapsto
            \quad
            \begin{array}{c}
              \infer[\lright]{S \mid \GG , \GD \vdG A' \lolli B'}{
              \infer[\scut]{S \mid \GG , \GD , A' \vdG B'}{
                \infer[\lright]{S \mid \GG \vdG A \lolli B}{
                  \deduce{S \mid \GG , A \vdG B}{f'}
                }
                &
              \deduce{A \lolli B \mid \GD , A' \vdG B'}{g'}
              }   
            }
            \end{array}
          \end{array}
        \end{displaymath}
      \end{itemize}
For $\mf{ccut}$, if $g = \lright \ g'$, then we permute $\mf{ccut}$ up.
If $g = \lleft (g' ,g'')$, we permute $\mf{ccut}$ up as well, but depending on where the cut formula is placed, we either apply $\mf{ccut}$ on $f$ and $g'$ or $f$ and $g''$.
% If $g = \lleft (g' , g'')$ and $f = \lright f'$, then depending on where the cut formula being placed, 
% \begin{displaymath}
%   \begin{array}{c}
%     \infer[\mf{ccut}]{A' \lolli B' \mid \GD_0 , \GG , \GD_1 \vdG C}{
%       \infer[\lright]{{-} \mid \GG \vdG A \lolli B}{
%         \deduce{{-} \mid \GG , A \vdG B}{f'}
%       }
%       &
%       \infer[\lleft]
%     }
%   \end{array}
% \end{displaymath}
\end{proof}
Moreover, \LSkG~is sound and complete wrt. left skew monoidal closed categories \cite[Theorem 3.2]{UVW:protsn}.

By soundness and completeness, similar to the result in \cite{uustalu:sequent:2021} for skew monoidal categories, we mean that \LSkG~is deductively equivalent to the axiomatic characterization of the free left skew monoidal closed category.
% This axiomatic calculus (\LSkNL) is reminiscent of Lambek's syntactic calculus \cite{lambek:deductive:68,lambek:deductive:69}:
\begin{equation*}\label{eq:seqcalc:skmc:Lam}
  \begin{array}{c}
        \infer[\id]{A \vdL A}{}
        \qquad
        \infer[\mathsf{comp}]{A \vdL C}{
          A \vdL B
          &
          B \vdL C
        }
%    \end{displaymath}
        %    \begin{displaymath}
        \qquad
      \infer[\otimes]{A \ot B \vdL C \ot D}{
        A \vdL C
        &
        B \vdL D
      }
      \qquad
      \infer[\lolli]{A \lolli B \vdL C \lolli D}{
        C \vdL A
        &
        B \vdL D
      }
%    \end{displaymath}
      %    \begin{displaymath}
      \\[5pt]
      \infer[\lambda]{\I \ot A \vdL A}{}
      \quad
      \infer[\rho]{A \vdL A \ot \I}{}
      \quad
      \infer[\alpha]{(A \ot B) \ot C \vdL A \ot (B \ot C)}{}
%    \end{displaymath}
      %    \begin{displaymath}
      \quad
      \infer=[\pi]{A \vdL B \lolli C}{A \ot B \vdL C}
      % \quad
      % \infer[\pi^{-1}]{A \ot B \vdL C}{A \vdL B \lolli C}
  \end{array}
\end{equation*}
% We call $\lambda$, $\rho$, and $\alpha$ structural axioms.
In particular, this is a semi-unital and semi-associative variation of Moortgat and Oehrle's calculus \cite[Chapter 4]{moot:categorial:2012} of non-associative Lambek calculus (\NL), where only right residuation is present.
We only care about sequent derivability in this section, therefore we omit the congruence relations on sets of derivations $A \vdL B$ and $S \mid \Gamma \vdG A$ that identify certain pairs of derivations.
However, the congruence relations are essential for these calculi being correct characterizations of the free left skew monoidal closed category.

The calculus \LSkG, being an equivalent presentation of a skew version of \NL,\  provides an effective procedure to determine formulae derivability in \LSkNL.
In other words, for any formula $A$, $ \vdL A$ if and only if ${-} \mid \quad \vdG A$.
Exhaustive proof search in \LSkG~always terminates, so for any $A$, either it finds a proof or it fails and there is no proof
% there is a proof in \LSkG~or the proof search terminates without finding any proof.

Adapted from \cite{moot:categorial:2012}, we define trees inductively by the grammar $T::= \mf{Fma} \mid {-} \mid(T , T)$, where ${-}$ is an empty tree.
A context is a tree with a hole defined recursively as $\mc{C} ::= [\cdot] \mid ( \mc{C}, T ) \mid (T, \mc{C})$. 
The substitution of a tree into a hole is defined recursively:
\begin{displaymath}
  \begin{array}{rcl}
  subst([\cdot], U) &=& U
  \\
  subst((T',\mc{C}), U) &=& (T' , subst(\mc{C},U) )
  \\
  subst((\mc{C},T'), U) &=& (subst(\mc{C},U),T' )
  \end{array}
\end{displaymath}
We use $T[\cdot]$ to denote a context and $T[U]$ to abbreviate $subst(T[\cdot], U)$.
Sometimes we omit parentheses for trees when it does not cause ambiguity.
Sequents in \LSkT~are in the form $T \vdT A$ where $T$ is a tree and $A$ is a single formula.
\\
Derivations in \LSkT~are generated recursively by following rules:
\begin{equation*}\label{eq:seqcalc:skmc:Lam:tree}
  \begin{array}{lc}
  &
  \infer[\ax]{A \vdT A}{}
  \\[5pt]
  (\text{logical rules})
    &
    \quad
    \infer[\unitl]{T[\I] \vdT C}{T [{-}] \vdT C}
    \quad
    \infer[\unitr]{{-} \vdT \I}{}
    \quad
    \infer[\tl]{T [A \ot B] \vdT C}{T [A , B] \vdT C}
    \quad
    \infer[\tr]{T , U \vdT A \ot B}{
      T \vdT A
      &
      U \vdT B  
    }
    \\
    &
    \infer[\lleft]{T[A \lolli B , U] \vdT C}{
      U \vdT A
      &
      T[B] \vdT C
    }
    \quad
    \infer[\lright]{T \vdT A \lolli B}{T , A \vdT B}
    \\[5pt]
    (\text{structural rules})
    &
    % \infer=[\comm]{T[B ; A] \vdT C}{T [A , B] \vdT C}
    % \quad
    \infer[\mf{assoc}]{T [(U_0 , U_1) , U_2] \vdT C}{T [U_0 , (U_1 , U_2)] \vdT C}
    \quad
    \infer[\mf{unitL}]{T [{-}, U] \vdT C}{T [U] \vdT C}
    \quad
    \infer[\mf{unitR}]{T[U] \vdT C}{T[U , {-}] \vdT C}
  \end{array}
\end{equation*}
This calculus is similar to the ones for \NL~\cite{moot:categorial:2012} and \NL~with unit \cite{bulinska:2009} but with semi-associative ($\mf{assoc}$) and semi-unital ($\mf{unitL}$ and $\mf{unitR}$) rules.
The structural rule $\mf{unitL}$, read bottom-up, removes an empty tree from the left. It helps us to correctly characterize the axiom $\lambda$ in \LSkT, i.e. $\I \ot A \vdT A$ is derivable while $A \vdT \I \ot A$ is not.
Analogously for the rule $\mf{unitR}$, from a bottom-up perspective, adds an empty tree from the right, and we cannot capture $\rho$ in \LSkT~without $\mf{unitR}$ (a double question mark $??$ means that there is no rule can be applied):
\begin{displaymath}
  \begin{array}{c}
    \infer[\tl]{\I \ot A \vdT A}{
      \infer[\unitl]{\I , A \vdT A}{
        \infer[\mf{unitL}]{{-} , A \vdT A}{
          \infer[\ax]{A \vdT A}{}
        }
      }
    }
    \qquad
    \infer[\mf{unitR}]{X \vdT \I \ot X}{
      \infer[\tr]{X,{-} \vdT \I \ot X}{
        \deduce{X \vdT \I}{??}
        &
        \deduce{{-} \vdT X}{??}
      }
    }
    % \deduce{X \vdT \I \ot X}{??}
    \qquad
    \infer[\mf{unitR}]{A \vdT A \ot I}{
      \infer[\tr]{A , {-} \vdT A \ot \I}{
        \infer[\ax]{A \vdT A}{}
        &
        \infer[\unitr]{{-} \vdT \I}{}
      }
    }
    \qquad
    \infer[\tl]{X \ot \I \vdT X}{
      \infer[\unitl]{X , \I \vdT X}{
        \deduce{X , {-} \vdT X}{??}
      }
    }
  \end{array}
\end{displaymath}
% Admissibility of following rules in \LSkT~are proved by induction on height of derivations:
% \begin{displaymath}
%   \small\begin{array}{c}
%     \infer[\tl^{-1}]{T[A , B] \vdT C}{
%       \deduce{T[A \ot B] \vdT C}{f}
%     }
%     \qquad
%     \infer[\unitl^{-1}]{T[{-}] \vdT C}{
%       \deduce{T[\I] \vdT C}{f}
%     }
%   \end{array}
% \end{displaymath}
\begin{theorem}\label{thm:cut:adm:LSkT}
  \LSkT~is cut-free, i.e. the rule
  \begin{displaymath}
    \begin{array}{c}
    \infer[\mf{cut}]{T[U] \vdT C}{
      \deduce{U \vdT A}{f}
      &
      \deduce{T[A] \vdT C}{g}
    }
    \end{array}
  \end{displaymath}
  is admissible.
\end{theorem}
\begin{proof}
  % Proof proceeds by induction on the height of derivations and complexity of formulae.
  We perform induction on the structure of derivation $f$ of the left premise, and if necessary, we perform subinduction on the derivation $g$ or the complexity of the cut formula $A$.
  Cases of logical rules $\ax, \tl, \tr, \lleft$, and $\lright$ have been discussed in \cite{moot:categorial:2012}, so we only elaborate on the new cases arising in \LSkT.
  \begin{itemize}
    \item The first new case is that $f = \unitr$, then we inspect the structure of $g$.
    \begin{itemize}
      \item If $g = \ax : \I \vdT \I$, then we define $\mf{cut} (\unitr , \ax) = \unitr$.
      \item If $g = \unitl \ g'$, then there are two subcases:
      \begin{itemize}
        \item if the $\I$ introduced by $\unitl$ is the cut formula, then we define
        \begin{displaymath}
          \begin{array}{c}
          \infer[\mf{cut}]{T[{-}] \vdT C}{
              \infer[\unitr]{{-} \vdT \I}{}
              &
              \infer[\unitl]{T[\I] \vdT C}{
                \deduce{T[{-}] \vdT C}{g'}
              }
            }
          \end{array}
          \quad
          \mapsto
          \quad
          \begin{array}{c}
          \deduce{T[{-}] \vdT C}{g'}
          \end{array}
        \end{displaymath}
        \item if the $\I$ introduced by $\unitl$ is not the cut formula, then we define
        \begin{displaymath}
          \begin{array}{c}
            \infer[\mf{cut}]{T^{\{ \I := {-} \}}[\I] \vdT C}{
              \infer[\unitr]{{-} \vdT \I}{}
              &
              \infer[\unitl]{T[\I] \vdT C}{
                \deduce{T[{-}] \vdT C}{g'}
              }
            }
          \end{array}
          \quad
          \mapsto
          \quad
          \begin{array}{c}
            \infer[\unitl]{T^{\{ \I := {-} \}}[\I] \vdT C}{
              \infer[\mf{cut}]{T^{\{ \I := {-} \}}[{-}] \vdT C}{
                \infer[\ax]{{-} \vdT \I}{}
                &
                \deduce{T{-} \mid \vdT C}{g'}
              }
            }
          \end{array}
        \end{displaymath}
        where $T^{\{ \I := {-} \}}[\cdot]$ means that a formula occurrence $\I$ at some fixed position in the context has been replaced by ${-}$.
      \end{itemize}
      \item If $g = \mc{R} \ g'$, where $\mc{R}$ is a one-premise rule other than $\unitl$, then $\mf{cut} (\unitr , \mc{R} \ g') = \mc{R} (\mf{cut} (\unitr , g'))$.
      \item The cases of an arbitrary two-premises rule are similar.
    \end{itemize}
    \item Other new cases ($\unitl$ and structural rules) are in the type of one-premise left rules, where we can permute $\mf{cut}$ up. 
    For example, if $f = \mf{unitL} \ f'$, then we define 
    % $\mf{cut}(\mf{unitL} \ f', g) = \mf{unitL} (\mf{cut} (f' , g))$, i.e.
    \begin{displaymath}
    \begin{array}{c}
      \begin{array}{c}
        \infer[\mf{cut}]{T[T'[{-} , U]] \vdT C}{
          \infer[\mf{unitL}]{T'[{-}, U] \vdT A}{
            \deduce{T'[U] \vdT A}{f'}
          }
          &
          \deduce{T[A] \vdT C}{g}
        }
      \end{array}
      \quad
      \mapsto
      \quad
      \begin{array}{c}
        \infer[\mf{unitL}]{T[T'[{-} , U]] \vdT C}{
          \infer[\mf{cut}]{T[T'[U]] \vdT C}{
            \deduce{T'[U] \vdT A}{f'}
            &
            \deduce{T[A] \vdT}{g}
          }
        }
      \end{array}
    \end{array}
    \end{displaymath}
    The other cases are similar.
  \end{itemize}
\end{proof}
The proof of equivalence relies on the following admissible rule, lemma and definition.
\begin{displaymath}
  \begin{array}{c}
    \infer[\unitl^{-1}]{T[{-}] \vdT C}{T[\I] \vdT C}
  \end{array}
\end{displaymath}
% \begin{lemma}
%   The rule $\tl$ is invertible in \LSkG, i.e. this rule
%   \begin{displaymath}
%     \infer[\tl^{-1}]{A \mid B , \GG \vdG C}{
%       \deduce{A \ot B \mid \GG \vdG C}{f}
%     }
%   \end{displaymath}
%   is admissible.
% \end{lemma}
% \begin{proof}
%   By structural induction on $f$.
% \end{proof}
\begin{lemma}\label{lem:subst:T2G}
  Given a context $T[\cdot]$ and a derivation $f: A \mid \quad \vdG B$, there exists a derivation $f^*: T[A]^* \mid \quad \vdG T[B]^*$, where $T^*$ transforms a tree into a formula by replacing commas with $\ot$ and ${-}$ with $\I$, respectively.
  % If $f: T[B]^* \mid \quad \vdG C$ and $g: A \mid \quad \vdG B$, then $T[A]^* \mid \quad \vdG C$, where $T^*$ transforms a tree into a formula by replacing commas with $\ot$ and ${-}$ with $\I$, respectively.
\end{lemma}
\begin{proof}
  % By induction on the structure of $T[A]^*$ to construct $T[A]^* \mid \quad \vdG T[B]^*$ from $g$.
  
    % Notice that $\mf{scut}$ is admissible in \LSkG, so we only have to check whether $T[A]^* \mid \quad \vdG T[B]^*$ is derivable from $g$.
  Proof proceeds by induction on the structure of $T[\cdot]$.
  \\
  If $T[\cdot] = [\cdot]$, then we have $T[A]^* = A$ and $T[B]^* = B$, and $f : A \mid \quad \vdG B$ by assumption.
  \\
  If $T[\cdot] = T'[\cdot] , T''$, then by inductive hypothesis, we have $f^*: T'[A]^* \mid \quad \vdG T'[B]^*$ and following derivation:
  \begin{displaymath}
    \infer[\tl]{T'[A]^* \ot T''^* \mid \quad \vdG T'[B]^* \ot T''^*}{
      \infer[\tr]{T'[A]^* \mid T''^* \vdG T'[B]^* \ot T''^*}{
        \deduce{T'[A]^* \mid \quad \vdG T'[B]^*}{f^*}
        &
        \infer[\pass]{{-}\mid T''^*  \vdG T''^*}{
          \infer[\ax]{T''^* \mid \quad \vdG T''^*}{}
        }
      }
    }
  \end{displaymath} 
  The other case ($T[\cdot] = T'', T'[\cdot]$) is symmetric.
\end{proof}
\begin{defn}
  We define an encoding function $\ldbc {-} \mid {-} \rdbc$ that transforms a tree and an ordered list of formulae into a tree associating to the left:
\begin{displaymath}
\begin{array}{c}
  \ldbc T \mid [\ ] \rdbc = T
  \\[5pt]
  \ldbc T \mid B , \GG \rdbc = \ldbc (T , B) \mid \GG\rdbc
\end{array}
\end{displaymath}
\end{defn}
With the above lemmata, definition, and functions $s(S)$ that maps a stoup to a formula (i.e. $s(S) = I$ if $S = {-}$ or $s(S) = B$ if $S = B$) and $T^*$ that transforms trees into formulae, we can state and prove the equivalence between \LSkG~and \LSkT.
\begin{theorem}\label{thm:equiv:LSkGLSKT}
The calculi \LSkG~and \LSkT~are equivalent, meaning that the two statements below are true:
  \begin{itemize}
    \item For any derivation $f: S \mid \Gamma \vdG C$, there exists a derivation ${\mf{G2T}} f : \ldbc s(S) \mid \GG \rdbc \vdT C$.
    \item For any derivation $f: T \vdT C$, there exists a derivation $\mf{T2G} f: T^* \mid \quad \vdG C$.
  \end{itemize}
\end{theorem}
\begin{proof}
  Both $\mf{G2T}$ and $\mf{T2G}$ are proved by induction on height of $f$.

  For $\mf{G2T}$, the interesting cases are $\tr$ and $\lleft$.
  For example, if $f = \tr (f', f'')$, then by inductive hypothesis, we have two derivations $\mf{G2T} \ f': \ldbc s(S) \mid \GG \rdbc \vdT A$ and $\mf{G2T} \ f'': \ldbc \I \mid \GD \rdbc \vdT B$.
  Our goal sequent is $\ldbc\ldbc s(S) \mid \GG \rdbc \mid \GD \rdbc \vdT A \ot B$, which is constructed as follows:
  \begin{displaymath}
    \begin{array}{c}
      \infer[\mf{unitR}]{\ldbc\ldbc s(S) \mid \GG \rdbc \mid \GD \rdbc \vdT A \ot B}{
        \infer[\unitl^{-1}]{\ldbc\ldbc s(S) \mid \GG \rdbc , {-} \mid \GD \rdbc \vdT A \ot B}{
          \infer[\mf{assoc}^*]{\ldbc\ldbc s(S) \mid \GG \rdbc , \I \mid \GD \rdbc \vdT A \ot B}{
            \infer[\tr]{\ldbc s(S) \mid \GG \rdbc, \ldbc \I \mid \GD \rdbc \vdT A \ot B}{
              \deduce{\ldbc s(S) \mid \GG \rdbc \vdT A}{\mf{G2T} \ f'}
              &
              \deduce{\ldbc \I \mid \GD \rdbc \vdT B}{\mf{G2T} \ f''}
            }
          }
        }
      }
    \end{array}
  \end{displaymath}
  where $\mf{assoc}^*$ means multiple applications of $\mf{assoc}$.
  The case of $\lleft$ is similar.

  For $\mf{T2G}$, the proof relies on Lemma \ref{lem:subst:T2G} heavily.
  For example, when $f = \mf{unitR} \ g$, where we have $g: T[U,{-}] \vdT C$.
  By inductive hypothesis, we have $\mf{T2G} \ g: T[U^*\ot \I]^* \mid \quad \vdG C$.
  % The sequent $U^* \mid \quad \vdG U^* \ot \I$ is also derivable in \LSkG, and then by Lemma \ref{lem:subst:T2G}, we obtain a derivation $T[U^*]^* \mid \quad \vdG T[U^*\ot \I]^*$ and construct the desired derivation as follows
  With Lemma \ref{lem:subst:T2G}, we construct the desired derivation as follows:
  \begin{displaymath}
    \begin{array}{c}
      \infer[\mf{scut}]{T[U^*]^* \mid \quad \vdG C}{
        \infer[\text{Lemma \ref{lem:subst:T2G}}]{T[U^*]^* \mid \quad \vdG T[U^*\ot \I]^*}{
          \infer[\tr]{U^* \mid \quad \vdG U^* \ot \I}{
            \infer[\ax]{U^* \mid \quad \vdG U^*}{}
            &
            \infer[\unitr]{{-} \mid \quad \vdG \I}{}
          }
        }
        &
        \deduce{T[U^*\ot \I]^* \mid \quad \vdG C}{\mf{T2G} \ g}
      }
    \end{array}
  \end{displaymath}
  The other cases are similar.
\end{proof}
\section{Skew Categories}\label{sec:skew:categories}
In this section, we present the definitions of left (right) skew monoidal closed categories, skew monoidal bi-closed categories, and various terms that will be used in the following section for discussion.
\begin{defn}\label{def:left:skewcat}
A \emph{left skew monoidal closed category} $\mathbb{C}$ is a category with a unit object $\I$ and two functors $\ot : \mathbb{C} \times \mathbb{C} \rightarrow \mathbb{C}$ and $\lolli : \mathbb{C}^{\mathsf{op}} \times \mathbb{C} \rightarrow \mathbb{C}$ forming an adjunction ${-} \ot B \dashv B \lolli {-}$ for all $B$,
and three natural transformations $\lambda$, $\rho$, $\alpha$ typed
	%\begin{displaymath}
$\lambda_A : \I \ot A \to A$, $\rho_A : A \to A \ot \I$ and $\alpha_{A,B,C} : (A \ot B) \ot C \to A \ot (B \ot C)$,
	%\end{displaymath}
satisfying coherence conditions on morphisms due to Mac Lane \cite{maclane1963natural}:
	\begin{center}
	  %(m1)
	  % https://q.uiver.app/?q=WzAsMyxbMSwwLCJcXEkgXFxvdCBcXEkiXSxbMCwxLCJcXEkiXSxbMiwxLCJcXEkiXSxbMSwwLCJcXHJob197XFxJfSJdLFswLDIsIlxcbGFtYmRhX3tcXEl9Il0sWzEsMiwiIiwyLHsibGV2ZWwiOjIsInN0eWxlIjp7ImhlYWQiOnsibmFtZSI6Im5vbmUifX19XV0=
	\begin{tikzcd}
		& {\I \ot \I} \\[-.2cm]
		\I && \I
		\arrow["{\rho_{\I}}", from=2-1, to=1-2]
		\arrow["{\lambda_{\I}}", from=1-2, to=2-3]
		\arrow[Rightarrow, no head, from=2-1, to=2-3]
	\end{tikzcd}
	\qquad
	%(m2)
	% https://q.uiver.app/?q=WzAsNCxbMCwwLCIoQSBcXG90IFxcSSkgXFxvdCBCIl0sWzEsMCwiQSBcXG90IChcXEkgXFxvdCBCKSJdLFsxLDEsIkEgXFxvdCBCIl0sWzAsMSwiQSBcXG90IEIiXSxbMywyLCIiLDAseyJsZXZlbCI6Miwic3R5bGUiOnsiaGVhZCI6eyJuYW1lIjoibm9uZSJ9fX1dLFszLDAsIlxccmhvX0EgXFxvdCBCIl0sWzEsMiwiQSBcXG90IFxcbGFtYmRhX3tCfSJdLFswLDEsIlxcYWxwaGFfe0EgLCBcXEkgLCBCfSJdXQ==
	\begin{tikzcd}
		{(A \ot \I) \ot B} & {A \ot (\I \ot B)} \\[-.3cm]
		{A \ot B} & {A \ot B}
		\arrow[Rightarrow, no head, from=2-1, to=2-2]
		\arrow["{\rho_A \ot B}", from=2-1, to=1-1]
		\arrow["{A \ot \lambda_{B}}", from=1-2, to=2-2]
		\arrow["{\alpha_{A , \I , B}}", from=1-1, to=1-2]
	\end{tikzcd}
	
	%(m3)
	% https://q.uiver.app/?q=WzAsMyxbMCwwLCIoXFxJIFxcb3QgQSApIFxcb3QgQiJdLFsyLDAsIlxcSSBcXG90IChBIFxcb3QgQikiXSxbMSwxLCJBIFxcb3QgQiJdLFswLDEsIlxcYWxwaGFfe1xcSSAsIEEgLEJ9Il0sWzEsMiwiXFxsYW1iZGFfe0EgXFxvdCBCfSJdLFswLDIsIlxcbGFtYmRhX3tBfSBcXG90IEIiLDJdXQ==
	\begin{tikzcd}
		{(\I \ot A ) \ot B} && {\I \ot (A \ot B)} \\[-.3cm]
		& {A \ot B}
		\arrow["{\alpha_{\I , A ,B}}", from=1-1, to=1-3]
		\arrow["{\lambda_{A \ot B}}", from=1-3, to=2-2]
		\arrow["{\lambda_{A} \ot B}"', from=1-1, to=2-2]
	\end{tikzcd}
	\qquad
	%(m4)
	% https://q.uiver.app/?q=WzAsMyxbMCwwLCIoQSBcXG90IEIpIFxcb3QgXFxJIl0sWzIsMCwiQSBcXG90IChCIFxcb3QgXFxJKSJdLFsxLDEsIkEgXFxvdCBCIl0sWzAsMSwiXFxhbHBoYV97QSAsIEIsIFxcSX0iXSxbMiwxLCJBIFxcb3QgXFxyaG9fQiIsMl0sWzIsMCwiXFxyaG9fe0EgXFxvdCBCfSJdXQ==
	\begin{tikzcd}
		{(A \ot B) \ot \I} && {A \ot (B \ot \I)} \\[-.3cm]
		& {A \ot B}
		\arrow["{\alpha_{A , B, \I}}", from=1-1, to=1-3]
		\arrow["{A \ot \rho_B}"', from=2-2, to=1-3]
		\arrow["{\rho_{A \ot B}}", from=2-2, to=1-1]
	\end{tikzcd}
	
	%(m5)
	% https://q.uiver.app/?q=WzAsNSxbMCwwLCIoQVxcb3QgKEIgXFxvdCBDKSkgXFxvdCBEIl0sWzIsMCwiQSBcXG90ICgoQiBcXG90IEMpIFxcb3QgRCkiXSxbMiwxLCJBIFxcb3QgKEIgXFxvdCAoQyBcXG90IEQpKSJdLFsxLDEsIihBIFxcb3QgQikgXFxvdCAoQyBcXG90IEQpIl0sWzAsMSwiKChBIFxcb3QgKEJcXG90IEMpIFxcb3QgRCkiXSxbMCwxLCJcXGFscGhhX3tBICwgQlxcb3QgQyAsIER9Il0sWzEsMiwiQSBcXG90IFxcYWxwaGFfe0IgLCBDICxEfSJdLFszLDIsIlxcYWxwaGFfe0EgLEIgLENcXG90IER9IiwyXSxbNCwzLCJcXGFscGhhX3tBIFxcb3QgQiAsIEMgLCBEfSIsMl0sWzQsMCwiXFxhbHBoYV97QSAsIEIgLEN9IFxcb3QgRCJdXQ==
	\begin{tikzcd}
		{(A\ot (B\ot C)) \ot D} && {A \ot ((B \ot C) \ot D)} \\[-.2cm]
		{((A \ot B)\ot C) \ot D} & {(A \ot B) \ot (C \ot D)} & {A \ot (B \ot (C \ot D))}
		\arrow["{\alpha_{A , B\ot C , D}}", from=1-1, to=1-3]
		\arrow["{A \ot \alpha_{B , C ,D}}", from=1-3, to=2-3]
		\arrow["{\alpha_{A ,B ,C\ot D}}"', from=2-2, to=2-3]
		\arrow["{\alpha_{A \ot B , C , D}}"', from=2-1, to=2-2]
		\arrow["{\alpha_{A , B ,C} \ot D}", from=2-1, to=1-1]
	\end{tikzcd}
	\end{center}
	\end{defn}
Left skew monoidal closed category has other equivalent characterizations \cite{street:skew-closed:2013,uustalu:eilenberg-kelly:2020}, because natural transformations $(\lambda, \rho, \alpha)$ are in bijective correspondence with tuples of (extra)natural transformations $(j, i, L)$ typed $j_A : \I \to A \lolli A$, $i_A : \I \lolli A \to A$, and $L_{A,B,C} : B \lolli C \to (A \lolli B) \lolli (A \lolli C)$.
In particular, in a left skew \emph{non-monoidal} closed category, $(\lambda, \rho, \alpha)$ are not available and one has to work with $(j, i, L)$ and corresponding equations.
% \begin{displaymath}
%   \begin{array}{c}
%     j_A : \I \to A \lolli A
%     \qquad
%     i_A : \I \lolli A \to A
%     \\[5pt]
%     L_{A,B,C} : B \lolli C \to (A \lolli B) \lolli (A \lolli C)
%   \end{array}
% \end{displaymath}
% Moreover, $\alpha$ and $L$ are interdefinable with a natural transformation $\mathsf{p}$ typed $\mathsf{p}_{A , B , C} : (A \ot B) \lolli C \to \linebreak A \lolli (B \lolli C)$, embodying an internal version of the adjunction between $\ot$ and $\lolli$.

\begin{defn}\label{def:right:skewcat}
A \emph{right skew monoidal closed category} $(\mathbb{C}, \I, \ot, \lolli)$ is defined with the same objects and adjoint functors as a in left skew monoidal closed category but three natural transformations $\lambda^{\mf{R}}$, $\rho^{\mf{R}}$, $\alpha^{\mf{R}}$ are typed
	%\begin{displaymath}
$\lambda^{\mf{R}}_A : A \to \I \ot A$, $\rho^{\mf{R}}_A : A \ot \I \to A$ and $\alpha^{\mf{R}}_{A,B,C} : A \ot (B \ot C) \to (A \ot B) \ot C$.
	%\end{displaymath}
The equations on morphisms are analogous but modified to fit the definition.
\end{defn}
Similar to left skew monoidal closed categories, natural transformations $(\lambda^{\mf{R}}, \rho^{\mf{R}}, \alpha^{\mf{R}})$ are in bijective correspondence with tuples ($j^{\mf{R}}, i^{\mf{R}}, L^{\mf{R}}$) typed $j^{\mf{R}}_{A, B} : \mbb{C} (\I , A \lolli B) \to \mbb{C} (A , B)$, $i^{\mf{R}}_{A} : A \to \I \lolli A$, and $L^{\mf{R}}_{A,B,C,D} : \mbb{C} (A , B \lolli (C \lolli D)) \to \int^{X} X. \mbb{C} (A , X \lolli D) \times \mbb{C} (B , C \lolli X)$, where $\int^{X}$ is a coend, cf. \cite[Section 4]{uustalu:eilenberg-kelly:2020}, and $\mbb{C} (A, B)$ means the set of morphisms from $A$ to $B$.
% \begin{displaymath}
%   \begin{array}{c}
%   j^{\mf{R}}_{A, B} : \mbb{C} (\I , A \lolli B) \to \mbb{C} (A , B) \qquad i^{\mf{R}}_{A} : A \to \I \lolli A
%   \\[5pt]
%   L^{\mf{R}}_{A,B,C,D} : \mbb{C} (A , B \lolli (C \lolli D)) \to \exists X. \mbb{C} (A , X \lolli D) \times \mbb{C} (B , C \lolli X)
%   \end{array}
% \end{displaymath}
% The existential quantifier in $L^{\mf{R}}_{A, B, C, D}$ is a coend $\int^{X}$ in its original definition in \cite[Section 4]{uustalu:eilenberg-kelly:2020}. 
In parts of the next sections, where we only work with thin categories (for any two objects $A$ and $B$, $\mbb{C} (A , B)$ is either empty or a singleton set), it is safe to replace $\int^{X}$ with an existential quantifier.
\\
In the rest of the paper, we usually omit subscripts of natural transformations.
\begin{defn}\label{def:normal}
  A left skew monoidal closed category is
  \begin{itemize}
    \item[--] \emph{associative normal} if $\alpha$ is a natural isomorphism;
    \item[--] \emph{left unital normal} if $\lambda$ is a natural isomorphism;
    \item[--] \emph{right unital normal} if $\rho$ is a natural isomorphism.
  \end{itemize}
  The $(j, i, L)$ version is similar. The case of right skew monoidal closed categories is analogous.
\end{defn}

\begin{defn}\label{def:SkewBiC}
  A category $(\mbb{C}, \I , \otl, \llolli, \otr, \rlolli)$ is skew monoidal bi-closed (\SkBiC) if there exists a natural isomorphism $\gamma : A \otl B \to B \otr A$, $(\mbb{C}, \I , \otl, \llolli)$ is left skew monoidal closed such that right skew structural rules are dictated by the left skew ones via $\gamma$.
  %  and $(\mbb{C}, \I , \otr, \rlolli)$ is right skew monoidal closed.
  \\
  This definition combines concepts from skew bi-monoidal and bi-closed categories as introduced in \cite{uustalu:eilenberg-kelly:2020}.
\end{defn}
\begin{example}
  $\lambda^{\mf{R}}$ is defined as $\gamma \circ \rho$, diagrammatically:
% https://q.uiver.app/#q=WzAsNCxbMCwwLCJBIl0sWzAsMiwiQSJdLFsyLDIsIkEgXFxvdGwgXFxJIl0sWzIsMCwiXFxJIFxcb3RyIEEiXSxbMSwwLCIiLDAseyJsZXZlbCI6Miwic3R5bGUiOnsiaGVhZCI6eyJuYW1lIjoibm9uZSJ9fX1dLFswLDMsIlxcbGFtYmRhXntcXG1me1J9fSJdLFsxLDIsIlxccmhvIiwyXSxbMiwzLCJcXGdhbW1hIiwyXV0=
\[\begin{tikzcd}[ampersand replacement=\&]
	A \&\& {\I \otr A} \\
	\\
	A \&\& {A \otl \I}
	\arrow["{\lambda^{\mf{R}}}", from=1-1, to=1-3]
	\arrow[Rightarrow, no head, from=3-1, to=1-1]
	\arrow["\rho"', from=3-1, to=3-3]
	\arrow["\gamma"', from=3-3, to=1-3]
\end{tikzcd}\]
\end{example}
In contrast to the categorical model of associative Lambek calculus, the monoidal bi-closed category, we do not have both left ($\lo$) and right residuation ($\ol$), but instead have two right residuations corresponding to different tensor products.
However, with the natural isomorphism $\gamma$, and selecting a specific tensor, we can simulate both left and right residuations.

In the remainder of the paper, we will develop axiomatic and sequent calculi for \SkBiC~and explore its relational semantics.

\section{Calculi for \SkBiC}\label{sec:calculi:skbic}
By defining new formulae and adding rules in \LSkNL, we can have an axiomatic calculus \SkBiCA, where formulae ($\mf{Fma}$) are inductively generated by the grammar $A,B::= X \mid \I \mid A \otl B \mid A \llolli B \mid A \otr B \mid A \rlolli B$. $X$ and $\I$ adhere to the definitions provided in Section \ref{sec:syntax}, and $\otl$ and $\llolli$ ($\otr$ and $\rlolli$) represent left (right) skew multiplicative conjunction and implication, respectively.
\\
Derivations in \SkBiCA~are inductively generated by following rules: 
\begin{equation*}\label{eq:seqcalc:biskmc:Lam}
  \small\begin{array}{c}
        \infer[\id]{A \vdL A}{}
        \quad
        \infer[\mathsf{comp}]{A \vdL C}{
          A \vdL B
          &
          B \vdL C
        }
        \\[5pt]
      \infer[\otl]{A \otl B \vdL C \otl D}{
        A \vdL C
        &
        B \vdL D
      }
      \quad
      \infer[\llolli]{A \llolli B \vdL C \llolli D}{
        C \vdL A
        &
        B \vdL D
      }
      \quad
      % \infer[\otr]{A \otr B \vdL C \otr D}{
      %   A \vdL C
      %   &
      %   B \vdL D
      % }
      % \quad
      \infer[\rlolli]{A \rlolli B \vdL C \rlolli D}{
        C \vdL A
        &
        B \vdL D
      }
      \\[5pt]
      \infer[\lambda]{\I \otl A \vdL A}{}
      \quad
      \infer[\rho]{A \vdL A \otl \I}{}
      \quad
      \infer[\alpha]{(A \otl B) \otl C \vdL A \otl (B \otl C)}{}
      % \\[5pt]
      % \infer[\rho^{\mf{R}}]{A \otr \I \vdL A}{}
      % \quad
      % \infer[\lambda^{\mf{R}}]{A \vdL \I \otr A}{}
      % \quad
      % \infer[\alpha^{\mf{R}}]{A \otr (B \otr C) \vdL (A \otr B) \otr C}{}
      \\[5pt]
      \infer[\Gg]{A \otl B \vdL B \otr A}{}
      \quad
      \infer[\Gg^{-1}]{A \otr B \vdL B \otl A}{}
      \quad
      \infer=[\pi]{A \vdL B \llolli C}{A \otl B \vdL C}
      \quad
      \infer=[\pi^{\mf{R}}]{A \vdL B \rlolli C}{A \otr B \vdL C}
  \end{array}
\end{equation*}
For any $f : A \vdL B$ and $g : C \vdL D$, we define $f \otr g$ as $\gamma \circ (g \otl f) \circ \gamma^{-1}$.
$\lambda^{\mf{R}}$, $\rho^{\mf{R}}$, and $\alpha^{\mf{R}}$ are also derivable.
\\
Similar to the constructions in \cite{uustalu:sequent:2021,uustalu:proof:nodate,uustalu:deductive:nodate,veltri:coherence:2021,UVW:protsn}, \SkBiCA~generates the free \SkBiC~(\FSkBiC) over a set $\mf{At}$ in the following way:
\begin{itemize}
  \item[--] Objects of \FSkBiC~are formulae ($\mf{Fma})$.
  \item[--] Morphisms between formulae $A$ and $B$ are derivations of sequents $A \vdL B$ and identified up to the congruence relation $\doteq$:  
  \begin{equation*}\label{eq:doteq}
\arraycolsep=20pt
\scriptsize\begin{array}{lc}
\text{(category laws)} &
\id \circ f \doteq f
\qquad
f \doteq f \circ \id
\qquad
(f \circ g) \circ h \doteq f \circ (g \circ h)
\\[5pt]
\text{($\otl$ functorial)} &
\id \otl \id \doteq \id
\qquad 
(h \circ f) \otl (k \circ g) \doteq h \otl k \circ f \otl g
\\[5pt]
\text{($\llolli$ functorial)} &
\id \llolli \id \doteq \id
\qquad 
(f \circ h) \llolli (k \circ g) \doteq h \llolli k \circ f \llolli g
\\[5pt]
% \text{($\otr$ functorial)} &
% \id \otr \id \doteq \id
% \qquad 
% (h \circ f) \otr (k \circ g) \doteq h \otr k \circ f \otr g
% \\[5pt]
\text{($\rlolli$ functorial)} &
\id \rlolli \id \doteq \id
\qquad 
(f \circ h) \rlolli (k \circ g) \doteq h \rlolli k \circ f \rlolli g
\\[5pt]
&
\lambda \circ \id \otl f \doteq f \circ \lambda
\\
\text{($\lambda,\rho,\alpha$ nat. trans.)} &
\rho \circ f \doteq f \otl \id \circ \rho
\\
&
\alpha \circ (f \otl g) \otl h \doteq f \otl (g \otl h) \circ \alpha
\\[5pt]
&
\lambda \circ \rho \doteq \id
\qquad
\id \doteq \id \otl \lambda \circ \alpha \circ \rho \otl \id
\\
(\text{Mac Lane axioms}) &
\lambda \circ \alpha \doteq \lambda \otl \id
\qquad
\alpha \circ \rho \doteq \id \otl \rho
\\
&
\alpha \circ \alpha \doteq \id \otl \alpha \circ \alpha \circ \alpha \otl \id 
\\[5pt]
(\text{$\gamma$ isomorphism}) &
\gamma \circ \gamma^{-1} \doteq \id
\qquad
\gamma^{-1} \circ \gamma \doteq \id
\\[5pt]
&
\pi f \circ g \doteq \pi (f \circ (g \otl \id))
\qquad
\pi (f \circ g) \doteq (\id \llolli f) \circ \pi g
\\[1.5pt]
(\text{$\pi^{(\mf{R})}$ nat. trans.})
&
\pi (\id \otl f) \doteq (g \llolli \id) \circ \pi \id
\qquad
\pi^{\mf{R}} (\id \otr f) \doteq (g \rlolli \id) \circ \pi^{\mf{R}} \id
\\[1.5pt]
&
\pi^{\mf{R}} f \circ g \doteq \pi^{\mf{R}} (f \circ (g \otr \id))
\qquad
\pi^{\mf{R}} (f \circ g) \doteq (\id \rlolli f) \circ \pi^{\mf{R}} g
\\[5pt]
(\text{$\pi^{(\mf{R})}$ isomorphism})
&
\pi (\pi^{-1} f) \doteq f
\qquad
\pi^{-1} (\pi f) \doteq f
\qquad
\pi^{\mf{R}} (\pi^{\mf{R}-1} f) \doteq f
\qquad
\pi^{\mf{R}-1} (\pi^{\mf{R}} f) \doteq f
\end{array}
\end{equation*}
Notice that by the definition of $f \otr g$ and $\gamma$ being an isomorphism, $\gamma$ and $\gamma^{-1}$ are natural transformations.
For example, $\gamma \circ f \otl g \doteq \gamma \circ f \otl g \circ \id \doteq \gamma \circ f \otl g \circ \gamma^{-1} \circ \gamma = g \otr f \circ \gamma$.
Similarly, naturality of $(\lambda^{\mf{R}}, \rho^{\mf{R}}, \alpha^{R})$ and corresponding Mac Lane axioms hold as well.
\end{itemize}
Given a skew monoidal bi-closed category $\mbb{D} $ with function $G: \mf{At} \to \mbb{D}$, we can define functions $\overline{G}_0 : \mf{Fma} \to \mbb{D}_0$ ($\mbb{D}_0$ is the collection of objects in $\mbb{D}$) and $\overline{G}_1 : \mf{FSkMBiC (At)} (A , B) \to \mbb{D} (\overline{G}_0 (A) , \overline{G}_0 (B))$ by induction on complexity of formulae and height of derivations respectively.
This construction uniquely specifies a strict skew monoidal bi-closed functor $\overline{G} : \mf{FSkMBiC} \to \mbb{D}$ satisfying $\overline{G} (X) = G(X)$.
% However, we do not yet know whether the free category generated by the rules of \SkBiCA~on a set of atomic formulae is the free \SkBiC~or not.
% Therefore, in this paper, we only consider provability complete models rather than proof complete models.
% In Section \ref{sec:algebraic:relational:model}, we will show that \SkBiCA~is sound and complete with respect to thin \SkBiC.

However, it remains unclear how to construct a sequent calculus {\`a} la Girard for \SkBiC.
A simpler scenario to consider is the sequent calculus for right skew monoidal closed categories.
In this context, recalling Definition \ref{def:right:skewcat}, where natural transformations are in an opposite direction compared to left skew monoidal closed categories.
One approach is to propose a dual sequent calculus to \LSkG.
Here, sequents would be of the form $\GG \mid S \vdG A$, indicating a reversal of stoup and context, with all left rules applicable solely to the stoup.
We should think of the antecedents as trees associating to the right, structured as $(A_n, (\dots, (A_1, A_0)) \dots)$.
Nevertheless, $\rlolli$, by definition, is again a right residuation, implying that $\rlleft$ and $\rlright$ should resemble those in \LSkG.
This requirement then necessitates contexts to appear on the right-hand side of the stoup.

% Fortunately, sequent calculus in the style of \LSkT captures skew monoidal bi-closed categories nicely.
% We call the calculus \SkBiCT~in which formulae are inductively generated by the grammar $A,B::= X \mid \I \mid A \otl B \mid A \llolli B \mid A \otr B \mid A \rlolli B$, where $X$ and $\I$ are the same as in section (\ref{sec:syntax}), and $\otl$ and $\llolli$ ($\otr$ and $\rlolli$) are left (right) skew multiplicative conjunction and implication.
Fortunately, we can develop a sequent calculus, denoted as \SkBiCT, which is inspired by \LSkT~to characterize \SkBiC~categories.
Specifically, \SkBiCT~is an instantiation of Moortgat's multimodal Lambek calculus \cite{moortgat:multimodl:1996} with unit, semi-unital, and semi-associative structural rules.

Trees in \SkBiCT~are inductively defined by the grammar $T ::= \mf{Fma} \mid {-}\mid (T, T)\mid(T;T)$.
What we have defined are trees with two different ways of linking nodes: through the use of commas and semicolons, corresponding to $\otl$ and $\otr$, respectively.
Contexts and substitution are defined analogously to those of \LSkT.
Sequents are in the form $T \vdT A$ analogous to those in Section \ref{sec:syntax}.
% Contexts and substitution are defined analogously to those outlined in Section \ref{sec:syntax}.
\\
Derivations in \SkBiCT~are generated recursively by following rules:
\begin{displaymath}
  \footnotesize\begin{array}{lc}
    &
    \infer[\ax]{A \vdT A}{}
    \quad
    \infer[\unitr]{{-} \vdT \I}{}
    \quad
    \infer[\unitl]{T[\I] \vdT C}{T [{-}] \vdT C}
    \\[5pt]
    (\text{logical rules})
    &
    \textcolor{red}{
    \infer[\tll]{T [A \ot^{\mf{L}} B] \vdT C}{T [A , B] \vdT C}
    \quad
    \infer[\trl]{T , U \vdT A \ot^{\mf{L}} B}{
      T \vdT A
      &
      U \vdT B
    }
    }
    \quad
    \textcolor{blue}{
    \infer[\tlr]{T [A \ot^{\mf{R}} B] \vdT C}{T [A ; B] \vdT C}
    \quad
    \infer[\trr]{T ; U \vdT A \ot^{\mf{R}} B}{
      T \vdT A
      &
      U \vdT B
    }
    }
    \\[5pt]

    &
    \textcolor{red}{
    \infer[\llleft]{T[A \llolli B , U] \vdT C}{
      U \vdT A
      &
      T[B] \vdT C
    }
    \quad
    \infer[\llright]{T \vdT A \llolli B}{T , A \vdT B}
    }
    \quad
    \textcolor{blue}{
      \infer[\rlleft]{T[A \rlolli B ; U] \vdT C}{
      U \vdT A
      &
      T[B] \vdT C
    }
    \quad 
    \infer[\rlright]{T \vdT A \rlolli B}{T ; A \vdT B}
    }
    \\[5pt]
    (\text{structural rules})
    &
    \textcolor{red}{
    \infer[\assl]{T [(U_0 , U_1) , U_2] \vdT C}{T [U_0 , (U_1 , U_2)] \vdT C}
    }
    \quad
    \infer=[\comm]{T[U_1 ; U_0] \vdT C}{T [U_0 , U_1] \vdT C}
    \quad
    \textcolor{blue}{
    \infer[\assr]{T [U_0 ; (U_1 ; U_2)] \vdT C}{T [(U_0 ; U_1) ; U_2] \vdT C}
    }
    \\[5pt]
    
    &
    \textcolor{red}{
    \infer[\mf{unitL^{L}}]{T[{-},U] \vdT C}{T[U] \vdT C}
    \quad
    \infer[\mf{unitR^{L}}]{T[U] \vdT C}{T[U,{-}] \vdT C}
    }
    \quad
    \textcolor{blue}{
    \infer[\mf{unitL^{R}}]{T[U;{-}] \vdT C}{T[U] \vdT C}
    \quad
    \infer[\mf{unitR^{R}}]{T[U] \vdT C}{T[{-};U] \vdT C}
    }
  \end{array}
\end{displaymath}
We can think of these rules as originating from two separate calculi: \LSkT~(the red part with $\ax,\unitr$, and $\unitl$) and another for right skew monoidal closed categories (\RSkT, the blue part with $\ax,\unitr$, and $\unitl$), linked by $\comm$, in other words, we can mimic all the blue rules in the style of \LSkT~(only commas appear in antecedents) and vice versa. 
% the colored rules except $\rlleft$ and $\rlright$ are interdefinable with $\comm$.
For example, we can express $\tlr$, $\trr$ and $\rlleft$ in the style of \LSkT: 
\begin{displaymath}
  \small\begin{array}{c}
    \begin{array}{c}
      \infer[\tlr']{T [B \ot^{\mf{R}} A] \vdT C}{T [A , B] \vdT C}
    \end{array}
    =
    \begin{array}{c}
      \infer[\tlr]{T [B \ot^{\mf{R}} A] \vdT C}{
        \infer[\comm]{T [B ; A] \vdT C}{T[A , B] \vdT C}
      }
    \end{array}
    \quad
    \begin{array}{c}
      \infer[\trr']{U, T \vdT A \otr B}{
        T \vdT A
        &
        U \vdT B
      }
    \end{array}
    =
    \begin{array}{c}
      \infer[\comm]{U , T \vdT A \otr B}{
        \infer[\tlr]{T ; U \vdT A \otr B}{
          T \vdT A
          &
          U \vdT B
        }
      }
    \end{array}
    \\
    \begin{array}{c}
      \infer[\rlleft']{T[U , A \rlolli B] \vdT C}{
      U \vdT A
      &
      T[B] \vdT C
    }
    \end{array}
    \quad
    =
    \quad
    \begin{array}{c}
      \infer[\ot \mf{comm}]{T[U , A \rlolli B] \vdT C}{
        \infer[\rlleft]{T[A \rlolli B ; U] \vdT C}{
          \deduce{U \vdT A}{}
          &
          \deduce{T[B] \vdT C}{}
        }
      }
    \end{array}
  \end{array}
\end{displaymath}

\begin{theorem}
  Similar to \LSkT, $\mf{cut}$ is admissible in \SkBiCT.
\begin{displaymath}
  \begin{array}{c}
    \infer[\mf{cut}]{T[U] \vdT C}{
      U \vdT A
      &
      T[A] \vdT C
    }
  \end{array}
\end{displaymath}
\end{theorem}
\begin{proof}
  The proof proceeds similarly to that of Theorem \ref{thm:cut:adm:LSkT}. 
  In particular, the new rules ($\ot\mf{comm}$ and the structural rules in blue) are all one-premise left rules, allowing us to permute $\mf{cut}$ upwards.
\end{proof}
The equivalence between \SkBiCA~and \SkBiCT~can be proved by induction on height of derivations with a lemma similar to Lemma \ref{lem:subst:T2G} and the following admissible rules:
\begin{displaymath}
  \begin{array}{c}
    \infer[\tll^{-1}]{T[A , B] \vdL C}{T[A \otl B] \vdL C}
    \quad
    \infer[\llright^{-1}]{T , A \vdL B}{T \vdL A \llolli B}
    \quad
    \infer[\tlr^{-1}]{T[A ; B] \vdL C}{T[A \otr B] \vdL C}
    \quad
    \infer[\rlright^{-1}]{T ; A \vdL B}{T \vdL A \rlolli B}
  \end{array}
\end{displaymath}
\begin{theorem}\label{thm:equiv:SkBiC}
\SkBiCT~is equivalent to \SkBiCA, meaning that the following two statements are true:
  \begin{itemize}
    \item For any derivation $f:A \vdL C$, there exists a derivation $\mf{A2G} f : A \vdT C$.
    \item For any derivation $f:T \vdT C$, there exists a derivation $\mf{G2A} f : T^{\#} \vdL C$, where $T^{\#}$ transforms a tree into a formula by replacing commas with $\otl$ and semicolons with $\otr$, and ${-}$ with $\I$, respectively.
  \end{itemize}
\end{theorem}

\section{Relational Semantics of \SkBiCA~and Application}\label{sec:algebraic:relational:model}
% In this section, we introduce algebraic and relational semantics (via Kripke frames) for \SkBiCA.
% Moreover, the relational semantics for \SkBiCA~is modular in the sense that we can build up semantics for semi-substructural logics incrementally by including more structural conditions in the frame.
% With its modularity, we will have an algebraic proof of main theorems about the interdefinability of a series of skew categories in \cite{uustalu:eilenberg-kelly:2020}.

In this section, we present the relational semantics of \SkBiCA.
Furthermore, the relational semantics for \SkBiCA~is characterized modularly, allowing us to construct models for semi-substructural logics step by step by incorporating additional structural conditions into the frame.
The modularity allows us to provide an algebraic proof for the main theorems concerning the interdefinability of a series of skew categories as discussed in \cite{uustalu:eilenberg-kelly:2020}.

A preordered ternary frame with a special subset is $\langle W ,\leq ,\mbb{I}, \mbb{L}\rangle$, where $W$ is a set, $\leq$ is a preorder relation on $W$, $\mbb{I}$ is a downwards closed subset of $W$, and $\mbb{L}$ is an arbitrary ternary relation on $W$, where $\mbb{L}$ is upwards closed on the first two arguments and downwards closed on the last argument with respect to $\leq$.
% In the remainder of the paper, we simply write $W$ instead of $(W, \leq)$ when there is no ambiguity.
\begin{defn}\label{eq:relational:monoids}
We list properties of ternary relations which we will focus on.
  \begin{equation*}
  \small\begin{array}{ll}
    \text{Left Skew Associativity (LSA)} & \forall a,b,c,d,x \in W, \mbb{L}abx \ \& \ \mbb{L}xcd \longrightarrow \exists y \in W \ \text{such that} \  \mbb{L}bcy \ \& \ \mbb{L}ayd.
    \\
    \text{Left Skew Left Unitality (LSLU)} & \forall a, b \in W, e \in \mbb{I}, \mbb{L}eab \longrightarrow b \leq a.
    \\
    \text{Left Skew Right Unitality (LSRU)} & \forall a \in W, \exists e \in \mbb{I} \ \text{such that} \  \mbb{L}aea.
    \\
    \text{Right Skew Associativity (RSA)} & \forall a,b,c,d,x \in W, \mbb{L}bcx \ \& \ \mbb{L}axd \longrightarrow \exists y \in W \ \text{such that} \  \mbb{L}aby \ \& \ \mbb{L}ycd.
    \\
    \text{Right Skew Left Unitality (RSLU)} & \forall a \in W, \exists e \in \mbb{I} \ \text{such that} \  \mbb{L}eaa.
    \\
    \text{Right Skew Right Unitality (RSRU)} & \forall a,b \in W, e \in \mbb{I}, \mbb{L}aeb \longrightarrow b \leq a.
  \end{array}
\end{equation*}
\end{defn}
Given another ternary relation $\mbb{R}$, we define
\begin{displaymath}
  \begin{array}{cc}
    \text{$\mbb{LR}$-reverse} & \forall a,b,c \in W, \mbb{L}abc \longleftrightarrow \mbb{R}bac.
  \end{array}
\end{displaymath} 
The associativity and unitality conditions are adapted from the theory of relational monoids \cite{rosenthal:relational:monoids:1997} and relational semantics for Lambek calculus \cite{dosen:1992}.

An \SkBiCA~frame is a quintuple $\langle W , \leq, \mbb{I} , \mbb{L} , \mbb{R}\rangle$, where $\mbb{LR}$-reverse is satisfied, $\mbb{L}$ satisfies LSA, LSLU, LSRU, and $\mbb{R}$ automatically satisfies RSA, RSLU, RSRU because of $\mbb{LR}$-reverse.

Unlike studies in \NL~e.g. \cite{dosen:1992,moortgat:multimodl:1996,moot:categorial:2012}, where two associativity conditions simultaneously hold for a relation or not, we explore two relations where one satisfies LSA and the other satisfies RSA.
Another distinction from the existing studies on semantics for \NL~with unit \cite{bulinska:2009} (or non-commutative linear logic \cite{abrusci:noncommutative:1990}) is that while $W$ is commonly assumed to be an unital groupoid (or monoid in the case of linear logic), here, we should consider that the unit behaves differently for different relations. 

We denote the set of downwards closed subsets of $W$ as $\mc{P}_{\darr}(W)$.
\begin{defn}\label{eq:valuation:frame}
A  function $v: \mf{Fma} \to \mc{P}_{\darr}(W)$ on a \SkBiCA~frame is a valuation if it satisfies:
  \begin{equation*}
  \small\begin{array}{ll}
    v(\I) & = \mbb{I}
    \\
    v(A \otl B) &= \{c : \exists a \in v(A) , b \in v(B) , \ \mbb{L}abc  \}
    \\
    v(A \llolli B) &= \{c : \forall a \in v(A), b \in W , \ \mbb{L}cab \Rightarrow b \in v(B) \}
    \\
    v(A \otr B) &= \{c : \exists a \in v(A) , b \in v(B) , \ \mbb{R}abc  \}
    \\
    v(A \rlolli B) &= \{c : \forall a \in v(A), b \in W , \ \mbb{R}cab \Rightarrow b \in v(B) \}
  \end{array}
\end{equation*}
\end{defn}
We define a \SkBiCA~model to be a \SkBiCA~frame with a valuation function, i.e. $\langle W , \leq , \mbb{I} , \mbb{L} , \mbb{R} , v \rangle$.
A sequent $A \vdL B$ is valid in a model $\langle W , \leq, \mbb{I} , \mbb{L} , \mbb{R} , v \rangle$ if $v(A) \subseteq v(B)$ and is valid in a frame if for any $v$ for that frame, $v(A) \subseteq v(B)$.
\begin{theorem}[Soundness]\label{thm:sound}
  If a sequent $A \vdL B$ is provable in \SkBiCA~then it is valid in any \SkBiCA~model.
\end{theorem}
\begin{proof}
  The proof is adapted from \cite{dosen:1992,moot:categorial:2012}, where the cases of $\alpha$ and $\alpha^{\mf{R}}$ have been discussed.
  Therefore, we only elaborate on new cases arising in \SkBiCA.
  % For soundness, the proof proceeds by induction on the height of derivations, and there are four new cases in \SkBiCA:
  \begin{itemize}
    \item[--] If the derivation is the axiom $\lambda: \I \otl A \vdL A$, then for any \SkBiCA~model $\langle W , \mbb{I} , \mbb{L} , \mbb{R} , v \rangle$ and any $a \in v(\I \otl A)$, there exist $e \in \mbb{I}$, $a' \in v(A)$, and $\mbb{L}ea'a$. By LSLU, we know that $a \leq a'$, and then $a \in v(A)$. 
    \item[--] If the derivation is the axiom $\rho: A \vdL A \otl \I$, then for any \SkBiCA~model $\langle W , \mbb{I} , \mbb{L} , \mbb{R} , v \rangle$ and any $a \in v(A)$, by LSRU, there exists $e \in \mbb{I}$ such that $\mbb{L}aea$, which means that $a \in v(A \otl \I)$.
    \item[--] If the derivation is the axiom $\gamma : A \otl B \vdL B \otr A$, then for any \SkBiCA~model $\langle W , \mbb{I} , \mbb{L} , \mbb{R} , v \rangle$ and any $c \in v(A \otl B)$, there exist $a \in v(A)$ and $b \in v(B)$ such that $\mbb{L}abc$. By $\mbb{LR}$-reverse, we have $\mbb{R}bac$, therefore $c \in v(B \otr A)$.
    \item[--] The case of $\gamma^{-1}$ is similar.  
  \end{itemize}
\end{proof}
\begin{defn}
  The canonical model of \SkBiCA~is $\langle W , \leq , \mbb{I} , \mbb{L}, \mbb{R}, v \rangle$ where
  \begin{itemize}
    \item[--] $W = \mf{Fma}$ and $A \leq B$ if and only if $A \vdL B$, 
    \item[--] $\mbb{I} = v(\I)$, 
    \item[--] $\mbb{L} ABC $ if and only if $C \vdL A \otl B$,
    \item[--] $\mbb{R} ABC$ if and only if $C \vdL A \otr B$, and
    \item[--] $v(A) = \{ B \mid B \vdL A \text{ is provable in \SkBiCA} \}$.
    % \item The canonical ternary frame is $\langle \mf{Fma} , \mbb{I} , \mbb{L} , \mbb{R} \rangle$, where
  \end{itemize}
\end{defn}
\begin{lemma}
  The canonical model is a \SkBiCA~model.
\end{lemma}
\begin{proof}
~
  %   Therefore the goal remaining is to show that the canonical model $\langle \mf{Fma} , \vdL , \mbb{I} , \mbb{L}, \mbb{R}, v \rangle$ is a \SkBiCA~model.
  % \\
  \begin{itemize}
    \item[--] The set $(\mf{Fma} , \vdL)$ is a preorder because of the rules $\id$ and $\comp$, and the set $\mbb{I}$ is downwards closed because of $\comp$.
    The relations $\mbb{L}$ and $\mbb{R}$ are downwards closed on their last argument because of the rule $\comp$.
    They are upwards closed on their first two arguments due to the rules $\otl$ and $\otr$,  respectively.
    These facts ensure that $\langle \mf{Fma} , \vdL, \mbb{I} , \mbb{L}, \mbb{R}\rangle$ is a ternary frame.
    \item[--] We show two cases (LSRU and LSRU) of the proof that $\mbb{L} , \mbb{R}$ satisfy their corresponding conditions, while other cases are similar.
    \begin{itemize} 
    \item[(LSLU)] Given any two formulae $A$ and $B$, and $J \in \mbb{I}$ with $\mbb{L}JAB$, we have $J \vdL \I$, and $B \vdL J \otl A$, then we can construct $B \vdL A$ as follows:
    \begin{displaymath}
     \begin{array}{c}
        
                  \infer[\comp]{B \vdL A}{
                    \infer[\comp]{B \vdL \I \otl A}{
                      \deduce{B \vdL J \otl A}{}
                      &
                      \infer[\otl]{J \otl A \vdL \I \otl A}{
                        \deduce{J \vdL \I}{}
                        &
                        \infer[\id]{A \vdL A}{}
                      }
                    }
                    &
                    \infer[\lambda]{\I \otl A \vdL A}{}
        }
      \end{array}
    \end{displaymath}
    \item[(LSRU)] By the axiom $\rho$, for any formula $A$, we have $A \vdL A \otl \I$, i.e. $\mbb{L}AIA$.
  \end{itemize}
    \item[--] The valuation $v$ is downwards closed because of the rule $\comp$.
    The other conditions on connectives are satisfied by definition.
  \end{itemize}
  Therefore, $\langle \mf{Fma} , \vdL , \mbb{I} , \mbb{L}, \mbb{R}, v \rangle$ is a \SkBiCA~model.
\end{proof}
\begin{theorem}[Completeness]
  If $A \vdL B$ is valid in any \SkBiCA~model, then it is provable in \SkBiCA.
\end{theorem}
\begin{proof}
  If $A \vdL B$ is valid in any \SkBiCA~model, then it is valid in the canonical model, i.e. $v(A) \subseteq v(B)$ in the canonical model.
  From $A \vdL A$, by definition of $v$, we have $A \in v(A)$, and because $v(A) \subseteq v(B)$, we know that $A \in v(B)$, therefore $A \vdL B$.
\end{proof}
We show a correspondence between frame conditions and the validity of structural laws in frames.
\begin{restatable}{theorem}{main}\label{thm:main}
  For any ternary frame $\langle W , \leq, \mbb{I} , \mbb{L} , \mbb{R} \rangle$,
  \begin{displaymath}
    \small\begin{array}{ccccc}
      &  \mbb{LR} \text{-reverse holds} & \longleftrightarrow  & \gamma \ \text{and} \ \gamma^{-1} \text{valid} &
      
      % \subseteq W, A \otl B = B \otr A &
      \\
      \alpha^{(\mf{R})} \ \text{valid} \ & \longleftrightarrow & \text{LSA (RSA) holds} & \longleftrightarrow & L^{(\mf{R})} \ \text{valid}
      \\
      \lambda^{(\mf{R})} \ \text{valid} \ & \longleftrightarrow & \text{LSLU (RSLU) holds} & \longleftrightarrow & j^{(\mf{R})} \ \text{valid}
      \\
      \rho^{(\mf{R})} \ \text{valid} \ & \longleftrightarrow & \text{LSRU (RSRU) holds} & \longleftrightarrow & i^{(\mf{R})} \ \text{valid}
    \end{array}
  \end{displaymath}
\end{restatable}
\begin{proof}
  The first case is that $\mbb{LR}$-reverse holds if and only if $\gamma$ and $\gamma^{-1}$ are valid, i.e. $v(A \otl B) = v(B \otr A)$.
  \begin{itemize}
    \item[$(\longrightarrow)$] For any $x \in v(A\otl B) \subseteq W$, there exists $a \in v(A), b \in v(B)$ and $\mbb{L}abx$.
    By $\mbb{LR}$-reverse, we have $\mbb{R}bax$ meaning that $x \in v(B \otr A)$.
    The other way around is similar.
    \item[$(\longleftarrow)$] Suppose that for any $v,A,B$, we have $v(A \otl B) = v(B \otr A)$.
    Consider any $a, b, x \in W$ such that $\mbb{L}abx$.
    We take $v(A) = a \darr$ and $v(B) = b\darr$ for some $A, B \in \mf{At}$.
    By the definition of $v$ and assumption, $x$ belongs to $v(A \otl B)$ and $v(B \otr A)$, therefore $\mbb{R}bax$.
    The other direction is similar.
  \end{itemize}
  % For the rest of the proof, when defining a valuation on $A$, we always assume that $A \in \mf{At}$.
  % Next, we consider other cases.
    \begin{itemize}
    \item[$\lambda: $] LSLU holds if and only if $\lambda$ is valid.
    \begin{itemize}
      \item[$(\longrightarrow)$] This is similar to case of $\lambda$ in the proof of Theorem \ref{thm:sound}.
      \item[$(\longleftarrow)$] 
      Suppose that $\lambda$ is valid, i.e. for any $A$ and $v$, we have $v(\I \otl A)  \subseteq v(A)$.
      Consider any $a, b \in W$, $e \in \mbb{I}$ such that $\mbb{L}eab$.
      We take $v(A) = a \darr$ for some $A \in \mf{At}$.
      By $\mbb{L}eab$ and the assumption, we know that $b \in v(A)$, which means that $b \leq a$.
      % Suppose that $\mbb{L}$ satisfies the negation of LSLU, i.e. $\exists a,b \in W, e \in \mbb{I}$ such that $\mbb{L}eab \ \& \ a < b$.
      % We take $v(A) = a\darr$, where $a\darr ::= \{ b' \mid b' \leq a \}$.
      % Assume that for any $A$, $v(\mbb{I} \otl A)  \subseteq v(A) $, then in particular, by $\mbb{L}eab$, we know that $b \in v(A)$, meaning $b \leq a$, which is a contradiction. 
    \end{itemize}
    \item[$\rho : $] LSRU holds if and only if $\rho$ is valid.
    \begin{itemize}
      \item[$(\longrightarrow)$] This is similar to case of $\rho$ in the proof of Theorem \ref{thm:sound}.
      \item[$(\longleftarrow)$] 
      Suppose $\rho$ is valid, i.e. for any $A$ and $v$, $v(A) \subseteq v(A \otl \I)$.
      Consider any $a \in W$. 
      We take $v(A) = a \darr$ for some $A \in \mf{At}$.
      By the assumption, there exist $a' \in v(A)$ and $e \in \mbb{I}$ such that $\mbb{L}a'ea$.
      Because $\mbb{L}$ is upwards closed, we know that $\mbb{L}aea$.
      % Suppose that LSRU does not hold, then $\exists a \in W$ such that $\forall e \in \mbb{I}$, $\neg \mbb{L}aea$.
      % We take $v(A) = a\darr$.
      % By anti-monotonicity of $\mbb{L}$, we have $\neg \mbb{L}aea'$ for all $a' \in v(A)$, therefore $v(A) \nsubseteq v(\mbb{I} \otl A)$.
    \end{itemize}
    \item[$\alpha: $] LSA holds if and only if $\alpha$ is valid.
    \begin{itemize}
      \item[$(\longrightarrow)$] For any $s \in v((A \otl B) \otl C)$, there exists $a \in v(A), b \in v(B), x \in v(A \otl B), c \in v(C), \mbb{L}abx$, and $\mbb{L}xcs$. 
      By LSA, there exists $y \in W$ such that $\mbb{L}bcy$ and $\mbb{L}ays$, then by definition of $v$, $y \in v(B \otl C)$ and $s \in v(A \otl (B \otl C))$. 
      % Standard argument in the literature on non-associative Lambek calculus.  
      \item[$(\longleftarrow)$] 
      Suppose that $\alpha$ is valid, i.e. for any $A, B, C, v$, we have $v((A \otl B) \otl C) \subseteq v(A \otl (B \otl C))$.
      Consider any $a,b,x,c,d \in W$ such that $\mbb{L}abx$ and $\mbb{L}xcd$.
      We take $v(A) = a\darr, v(B) = b\darr$, $v(C) = c\darr$ for some $A, B,C \in \mf{At}$, then we know that $x \in v(A \otl B)$ and $d \in v((A \otl B) \otl B)$.
      By the assumption, $d$ belongs to $v(A \otl (B \otl C))$ as well, which means that there exist $a',b',y,c' \in W$ such that $\mbb{L}b'c'y$ and $\mbb{L}a'yd$.
      Because $\mbb{L}$ is upwards closed, we have $\mbb{L}bcy$ and $\mbb{L}ayd$ as desired.
    \end{itemize} 
    \item[$L:$] LSA holds if and only if for any $A,B,C$ and $v$, $v(B \llolli C) \subseteq v((A \llolli B) \llolli (A \llolli C))$. 
    \begin{itemize}
      \item[$(\longrightarrow)$] For any $s \in v(B \llolli C$), we show $s \in v((A \llolli B) \llolli (A \llolli C))$.
      By definition, from assumptions $x \in v(A \llolli B), \ \mbb{L}sxy, \ y \in v(A \llolli C),\  a \in A, \ c \in W$, and $\mbb{L}yac$, we have to prove that $c \in C$.
      By LSA, there exists $x' \in W$ such that $\mbb{L}xax'$ and $\mbb{L}sx'c$. 
      We get $x' \in B$ due to $x \in v(A \llolli B)$.
      Thus, we have $c \in C$ because $s \in v(B \llolli C)$.
      \item[$(\longleftarrow)$] 
      Suppose that for any $A,B,C$ and $v$, we have $v(B \llolli C) \subseteq v((A \llolli B) \llolli (A \llolli C))$.
      Consider $a,b,x,c,d \in W$ such that $\mbb{L}abx$ and $\mbb{L}xcd$.
      Take $v(A) = c\darr$, $v(B) = \{y \mid \mbb{L}bcy \}$, and $v(C) = \{ d' \mid \exists y \in v(B), \mbb{L}ayd' \}$ for some $A,B,C \in \mf{At}$.
      Given any $y \in v(B)$ and any $d' \in W$, if $\mbb{L}ayd'$, then by definition of $v(C)$, $d' \in v(C)$, therefore $a \in v(B \llolli C)$.
      By assumption, $a \in v((A \llolli B) \llolli (A \llolli C))$ as well, which means that, for any $b' \in v(A  \llolli B)$, $x' \in W$ $c' \in v(A)$ and $d' \in W$, if $\mbb{L}ab'x'$, then $x' \in v(A \llolli C)$, and if $\mbb{L}x'c'd'$, then $d' \in C$.
      By the definition of $v(B)$ and assumptions $\mbb{L}abx$ and $\mbb{L}xcd$,  we have $b \in v(A \llolli B)$, $x \in v(A \llolli C)$, therefore $d \in v(C)$, which means that there exists $y \in W$ such that $\mbb{L}bcy$ and $\mbb{L}ayd$.   
    \end{itemize}
    \item[$j^{\mf{R}}:$] RSLU holds if and only if for any $A,B$ and $v$, if $\mbb{I} \subseteq v(A \rlolli B)$, then $v(A) \subseteq v(B)$.
    \begin{itemize}
      \item[$(\longrightarrow)$] By RSLU, for all $a \in v(A)$, there exists $e \in \mbb{I}$ such that $\mbb{R}eaa$, then we have $a \in v(B)$ because $e \in v(A \rlolli B)$.
      \item[$(\longleftarrow)$] 
      Suppose that for any $A, B$ and $v$, if $\mbb{I} \subseteq v(A \rlolli B)$, then $v(A) \subseteq v(B)$.
      Consider any $a \in W$.
      We take $v(A) = a \darr$ and $v(B) = \{ b \mid \exists e \in \mbb{I}, \mbb{R}eab \}$ for some $A,B \in \mf{At}$.
      For any $e' \in \mbb{I}$, $a' \in v(A)$, and $b' \in W$, if $\mbb{R}e'a'b'$, then because $\mbb{R}$ is upwards closed, we have $b' \in v(B)$, which means $e' \in v(A\rlolli B)$.
      Therefore $\mbb{I} \subseteq v(A \rlolli B)$.
      From the assumption, we can now conclude that $v(A) \subseteq v(B)$.
      In particular, $a \in v(B)$, which means that there exists $e \in \mbb{I}$ such that $\mbb{R}eaa$.
    \end{itemize}
    \item[$L^{\mf{R}}:\ $] RSA holds if and only if for any $A,B,C,D$ and $v$, if $v(A) \subseteq v(B \rlolli (C \rlolli D))$ then there exists $X$ such that $v(A) \subseteq v(X \rlolli D)$ and $v(B) \subseteq v(C \rlolli X)$.
    \begin{itemize}
      \item[$(\longrightarrow)$] We expand the assumption first.
    \\
    For any $A, B , C, D$, $a \in v(A)$, and $b, z \in W$, if $b \in v(B)$ and $\mbb{R}abz$ then $z \in v(C \rlolli D)$ and for all $z \in v(C \rlolli D)$, for all $c, d \in W$ if $c \in v(C)$ and $\mbb{R}zcd$, then $d \in v(D)$.
    In other words, for any $z, d \in W$, if there are $a \in v(A)$, $b \in v(B)$, $c \in v(C)$, $\mbb{R}abz$, and $\mbb{R}zcd$, then $d \in v(D)$.
    
    We show that $B \otr C$ satisfies following two statements:
    \begin{itemize}
      \item[--] For any $a \in v(A)$, we show that $a \in v((B \otr C) \rlolli D)$. For any $x \in v(B \otr C)$ and $d \in W$, if $\mbb{R}axd$, then by definition of $\otr$, we have $\mbb{R}bcx$, where $b \in v(B)$ and $c \in v(C)$. By RSA, there exists $z \in W$ such that $\mbb{R}abz$, and $\mbb{R}zcd$. 
      By the expanded assumption, $d \in v(D)$. Therefore $a \in v((B \otr C) \rlolli D)$.
      \item[--] For any $b \in v(B)$, $c \in v(C)$, and $x \in W$, suppose $\mbb{R}bcx$, then $x \in v(B \otr C)$ by definition of $\otr$. Therefore $b \in v(C \rlolli (B \otr C))$.
    \end{itemize} 
    % In general, any $X = B' \otr C'$, where $B'\subseteq B$ and $C' \subseteq C$, satisfies the above statements as well.
    \item[$(\longleftarrow)$] 
    Assume that for any $A,B,C,D$ and $v$, if $v(A) \subseteq v(B \rlolli (C \rlolli D))$, then there exists $X$ such that $v(A) \subseteq v(X \rlolli D)$ and $v(B) \subseteq v(C \rlolli X)$.
    Suppose that we have $a,b,c,d,x\in W$ such that $\mbb{R}axd$ and $\mbb{R}bcx$, then we take $v(A) = a\darr$, $v(B) = b\darr$, $v(C) = c\darr$, and $v(D) = \{d' \mid \exists y,\mbb{R}aby \& \mbb{R}ycd' \}$ for some $A,B,C,D \in \mf{At}$.
    For any $a' \in v(A)$, given any $b' \in v(B)$, $x' \in W$, $c' \in v(C)$, $d' \in W$ such that $\mbb{R}a'b'x'$ and $\mbb{R}x'c'd'$.
    Because $\mbb{R}$ is upwards closed, by the definition of $v(D)$, we have $d' \in v(D)$, which means $v(A) \subseteq v(B \rlolli (C \rlolli D))$.
    By the assumption, there exists $X$ such that
    \begin{enumerate}[(1)]
      \item $v(A) \subseteq v(X \rlolli D)$, which means that for any $a' \in v(A)$, given any $x' \in X$, $d' \in W$, if $\mbb{R}a'x'd'$, then $d' \in v(D)$, and 
      \item $v(B) \subseteq v(C \rlolli X)$, which means that for any $b' \in v(B)$, given any $c' \in v(C)$ and $x' \in W$, if $\mbb{R}b'c'x'$, then $x' \in v(X)$.
    \end{enumerate}
    By $\mbb{R}bcx$, and $(2)$, we know that $x \in v(X)$.
    By $\mbb{R}axd$, and $(1)$, we know that $d \in v(D)$, which means that there exists $y \in W$ such that $\mbb{R}aby$ and $\mbb{R}ycd$.
    \end{itemize}
  \end{itemize}
  The other cases are similar to the arguments above.
  % The ($\Longrightarrow$) parts are straightforward while the ($\Longleftarrow$) parts need more care. 
\end{proof}
A frame $\langle W, \leq, \mbb{I}, \mbb{L}\rangle$ is left (right) skew associative if $\mbb{L}$ satisfies LSA (RSA).
For other conditions, the naming is similar.
If $\langle W, \leq, \mbb{I}, \mbb{L}\rangle$ satisfies LSA, LSLU, and LSRU (respectively RSA, RSLU, RSRU), then it is a left (respectively right) skew.

We can think of a \SkBiCA~frame $\langle W, \leq, \mbb{I}, \mbb{L}, \mbb{R}\rangle$ as a combination of two ternary frames $\langle W, \leq, \mbb{I}, \mbb{L}\rangle$ (left skew frame) and $\langle W, \leq, \mbb{I}, \mbb{R} \rangle$ (right skew frame) sharing the same set of possible worlds, where the ternary relations are interdefinable by
$\mbb{L}\mbb{R}$-reverse.
%  plays an important role in the picture.
% Intuitively, for any ternary frame $\langle W, {\leq}, \mbb{I}, \mbb{L} \rangle$, if there exists a relation $\mbb{R}$ such that 
Whenever $\mbb{L}\mbb{R}$-reverse holds, then $\langle W, {\leq}, \mbb{I}, \mbb{L} \rangle$ is left skew if and only if $\langle W, \leq, \mbb{I}, \mbb{R} \rangle$ is right skew.
In fact, we have:
\begin{equation*}
% \label{eq:normal:frame}
    \begin{array}{lcl}
      \langle W, \leq, \mbb{I}, \mbb{L}\rangle \ \text{left skew associative} & \longleftrightarrow & \langle W, \leq, \mbb{I}, \mbb{R}\rangle \ \text{right skew associative}
      \\
      \langle W, \leq, \mbb{I}, \mbb{L}\rangle \ \text{left skew left unital} & \longleftrightarrow & \langle W, \leq, \mbb{I}, \mbb{R}\rangle \ \text{right skew right unital}
      \\
      \langle W, \leq, \mbb{I}, \mbb{L}\rangle \ \text{left skew right unital} & \longleftrightarrow & \langle W, \leq, \mbb{I}, \mbb{R}\rangle \ \text{right skew left unital}
    \end{array}
  \end{equation*}
% We will prove the statement formally in the following subsection by restating the definitions and statements in (eilenberg-kelly) algebraically.
% We are working with thin categories, so we can ignore the coherence conditions in skew monoidal bi-closed categories and have:
% With this theorem, we have:
% \subsection{Algebraic Investigation on \SkBiC \ via Kripke Frames}

If we state the structural laws semantically rather than sequents, we can reformulate Theorem \ref{thm:main} without referring to sequents and valuations.
For example, we can define $\otl$ on downwards closed sets of worlds as $A \otl B = \{c : \exists a \in A \ \& \ b \in B \ \& \ \mbb{L}abc  \}$ and express $\alpha$ as $(A \otl B) \otl C \subseteq A \otl (B \otl C)$.
It is the case that $\alpha$ holds in a frame if and only if it satisfies LSA.

We construct a thin \SkBiC~from the frame $\langle W, \leq, \I, \mbb{L}, \mbb{R}\rangle$ and provide algebraic proofs for main theorems in \cite{uustalu:eilenberg-kelly:2020}.
The objects in the category are downwards closed subsets of $W$ and for $A, B$, we have a map $A \to B$ if and only if $A \subseteq B$.
\begin{corollary}
The category $(\mc{P}_{\darr}(W) , \subseteq)$ generated from any \SkBiCA~frame is a thin \SkBiC.
\end{corollary}
A frame $\langle W, \leq, \mbb{I}, \mbb{L}\rangle$ is associative normal if it satisfies LSA and RSA simultaneously, and left (right) unital normal if LSLU and RSLU (LSRU and RSRU) are satisfied.
Therefore, by Theorem \ref{thm:main}, we have a thin version of the main results in \cite{uustalu:eilenberg-kelly:2020}.
\begin{corollary}
  Given any frame, for the category $(\mc{P}_{\darr}(W), \subseteq)$ generated from the frame  we have:
  \begin{displaymath}
    \begin{array}{lcl}
       (\mbb{I}, \otl) \ \text{left skew monoidal} & \longleftrightarrow &  (\mbb{I}, \llolli) \ \text{left skew closed}
      \\
      (\mbb{I}, \otr) \ \text{right skew monoidal} & \longleftrightarrow &  (\mbb{I}, \rlolli) \ \text{right skew closed}
    \end{array}
  \end{displaymath}
  Moreover, if the frame satisfies $\mbb{LR}$-reverse then:
  \begin{displaymath}
    \begin{array}{lcl}
      (\mbb{I}, \otl) \ \text{left skew monoidal} & \longleftrightarrow &  (\mbb{I}, \otr) \ \text{right skew monoidal}
      \\
      (\mbb{I}, \llolli) \ \text{left skew closed} & \longleftrightarrow &  (\mbb{I}, \rlolli) \ \text{right skew closed}
      \\
      (\mbb{I}, \otl) \ \text{associative normal} & \longleftrightarrow &  (\mbb{I}, \otr) \ \text{associative normal}
      \\
      (\mbb{I}, \otl) \ \text{left unital normal} & \longleftrightarrow &  (\mbb{I}, \otr) \ \text{right unital normal}
      \\
      (\mbb{I}, \otl) \ \text{right unital normal} & \longleftrightarrow &  (\mbb{I}, \otr) \ \text{left unital normal}
      \\
      (\mbb{I}, \llolli) \ \text{associative normal} & \longleftrightarrow &  (\mbb{I}, \rlolli) \ \text{associative normal}
      \\
      (\mbb{I}, \llolli) \ \text{left unital normal} & \longleftrightarrow &  (\mbb{I}, \rlolli) \ \text{right unital normal}
      \\
      (\mbb{I}, \llolli) \ \text{right unital normal} & \longleftrightarrow &  (\mbb{I}, \rlolli) \ \text{left unital normal}
    \end{array}
  \end{displaymath}
  % \begin{itemize}
  %   \item $\langle W, \leq, $
  % \end{itemize}
\end{corollary}
\section{Concluding remarks}
This paper discusses sequent calculi for left (right) skew monoidal categories and skew monoidal bi-closed categories in the style of non-associative Lambek calculus.
Compared to the sequent calculi with stoup, although the calculi {\`a} la Lambek are not immediately decidable but are more flexible in the sense that the sequent calculi for right skew monoidal closed categories (\RSkT) and skew monoidal bi-closed categories (\SkBiCT) are presentable.
Moreover, we show that they are cut-free and equivalent to the calculus with stoup (Theorem \ref{thm:equiv:LSkGLSKT}) and the axiomatic calculus (Theorem \ref{thm:equiv:SkBiC}).

In the last section, we focus on the relational semantics of \SkBiCA~via the ternary frame \linebreak $\langle W, {\leq}, \mbb{I}, \mbb{L}, \mbb{R} \rangle$ where $\mbb{L}$ and $\mbb{R}$ are connected by $\mbb{LR}$-reverse and therefore if $\mbb{L}$ satisfies left skew structural conditions then $\mbb{R}$ satisfies right skew structural conditions automatically.
By Theorem \ref{thm:main}, for any \SkBiCA~model, we can construct a thin skew monoidal bi-closed category $(\mc{P}_{\darr}(W) , \subseteq)$.
In addition, we can obtain algebraic proofs of main theorems in \cite{uustalu:eilenberg-kelly:2020}.

A future project is to explore Craig interpolation \cite{craig:interpolation:1957} for semi-substructural logics.
In \LSkT, the situation is more complicated than either associative or fully non-associative Lambek calculi because we only allow semi-associativity.
Consider the statement:
\begin{itemize}
    \item[\ ] Given a derivation, $f : T[U] \vdT C$, then there exist a formula $D$ and two derivations $f_0 : U \vdT D$ and $f_1: T[D] \vdT C$, and $\mf{var} (D) \subseteq \mf{var} (\Gd (U)) \cap \mf{var}(\Gd (T[{-}]), C)$, where $\Gd$ is a function that transforms a tree into a list of formulae.
\end{itemize}
If we try to prove by induction on $f$, then there is a critical case
\begin{displaymath}
  \small\begin{array}{c}
    \infer[\mf{assoc}]{T[(U_0, U_1), U_2] \vdT C}{
      \deduce{T[U_0, (U_1, U_2)] \vdT C}{f}
    }
  \end{array}
\end{displaymath}
where $U = U_0, U_1$, therefore, the goal is to find a formula $D$ and two derivations $g: U_0, U_1 \vdT D$ and $T[D,U_2] \vdT C$.
However, we cannot directly apply the inductive hypothesis twice on $f$, because the procedure of finding an interpolant formula and corresponding derivations is not height preserving.
Therefore, proving the interpolation property for semi-substructural logics is more subtle than expected.

Another possible direction is to incorporate modalities (exponentials in linear logical terminology) with semi-substructural logic as in \cite{moortgat:multimodl:1996} (modalities) and \cite{Blaisdell:subexpon:linear:2022} (subexponentials) with non-associative Lambek calculus and non-commutative and non-associative linear logic.

Similar to the equational theories for \SkBiCA~discussed in Section \ref{sec:calculi:skbic}, we also plan to investigate the equational theories on the derivations of \LSkT~and \SkBiCT~in the future.

\paragraph{Acknowledgements} 
We thank Giulio Fellin, Tarmo Uustalu, and Niccol{\`o} Veltri for invaluable discussions and the anonymous reviewers for constructive feedback and comments.
Special thanks to Tarmo Uustalu and Niccol{\`o} Veltri for thorough review, for highlighting some inaccuracies in the draft, and their assistance in resolving these issues.
This work was supported by the Estonian Research Council grant PSG749.
%  and the ESF funded Estonian IT Academy research measure (project 2014-2020.4.05.19-0001). 
\bibliographystyle{eptcs}
\bibliography{ncl24}
\end{document}